\documentclass[11pt]{article}

\usepackage[utf8]{inputenc}
\usepackage[T1]{fontenc}
\usepackage[british]{babel}
\usepackage[a4paper]{geometry}
\usepackage{amsmath,amssymb,amsthm}
\usepackage[shortlabels]{enumitem}
\usepackage[all]{xy}
\usepackage{tikz-cd}
\usepackage{authblk}
\usepackage{xcolor}

\theoremstyle{plain}
\newtheorem{theo}{Theorem}[section]
\newtheorem{lem}[theo]{Lemma}
\newtheorem{cor}[theo]{Corollary}
\newtheorem{prop}[theo]{Proposition}
\theoremstyle{definition}
\newtheorem{defi}[theo]{Definition}

\newtheorem{rem}[theo]{Remark}

\theoremstyle{plain}

\newtheorem*{prop_sans_num}{Proposition}

\theoremstyle{definition}
\newtheorem*{defi_sans_num}{Definition}
\newtheorem*{rem_sans_num}{Remark}

\theoremstyle{plain}
\newtheorem{theorem}{Theorem}

\def\shaun#1{\textcolor{blue}{#1}}

\title{Universal Weil module}
\author{Justin Trias}
\date{}

\begin{document}

\maketitle

\begin{abstract} The classical construction of the Weil representation, with complex coefficients, has long been expected to work for more general coefficient rings. This paper exhibits the minimal ring $\mathcal{A}$ for which this is possible, the integral closure of $\mathbb{Z}[\frac{1}{p}]$ in a cyclotomic field, and carries out the construction of the Weil representation over $\mathcal{A}$-algebras. As a leitmotif all along the work, most of the problems can actually be solved over the base ring $\mathcal{A}$ and transferred to any $\mathcal{A}$-algebra by scalar extension. The most striking fact is that all these Weil representations arise as the scalar extension of a single one with coefficients in $\mathcal{A}$. In this sense, the Weil module obtained is universal. Building upon this construction, we speculate and make predictions about an integral theta correspondence. \end{abstract}

\section*{Introduction}

This paper is intended as a stepping-stone in the direction of an ``integral theta correspondence''. Whatever this may be, it will require a theory of the Weil representation over rings and the purpose of this paper is to carry this out on rings with minimal hypotheses. When the coefficient ring is the field of complex numbers, this representation originated in problems related to $\theta$-series and was first constructed in the seminal paper \cite{weil_article} of André Weil. 

There is another way, as opposed to the original approach of Weil, to build this representation. Because of its relations with quantum physics, it appears often in older literature as the so-called oscillator representation and involves the famous Stone-von Neumann theorem as a cornerstone in this alternative construction \cite{howe}. It plays a pivotal role in the theta correspondence, where the interplay between this representation and the dual pairs introduced by Roger Howe \cite{howe} led to a conjectural bijective correspondence between some subsets of irreducible representations for each member of the dual pair. 

This correspondence, known in older literature as \emph{Howe duality} or \emph{the Howe correspondence}, took almost 40 years to be completely proven, and is now usually known as \emph{the theta correspondence}. The main works which led to its proof include \cite{howe,rallis,kudla_invent,mvw,wald,minguez_howe,gt,gan_sun} and we refer to \cite{trias_theta1} for a more detailed exposition of these contributions to the classical theta correspondence. This celebrated bijection plays a central role in number theory as it encodes a lot of arithmetic information and allows one to build automorphic forms. It is the centre of a highly active research field in the topic.

In the 1980's, Marie-France Vignéras studied, in relation to Serre's conjectures, congruences between automorphic representations by means of the modular representation theory of their local factors. She considered smooth representations of connected reductive $p$-adic groups with coefficients in fields that are more general than the complex numbers, allowing in particular fields of positive characteristic. The theory splits into two different aspects, depending on whether the characteristic of the coefficient is different from $p$ or not. In the first case, which we study here, we talk about $\ell$-modular representations by implicitly meaning that $\ell \neq p$. (The second case is referred to as modulo $p$ representation theory and requires complete different techniques.)

An important result about these $\ell$-modular representations is the compatibility of the classical local Langlands correspondence for general linear groups with a certain $\ell$-modular one as described in \cite{vig_corresp_Langlands_ss_mod_ell}. In recent years, there has been a growing interest in studying representations in families \textit{i.e.} over coefficient rings where $p$ is invertible. For general linear groups, families of representations with coefficients in a Witt ring $W(\overline{\mathbb{F}_\ell})$ are quite well understood \cite{helm_bernstein_center} and provides a local Langlands correspondence in families \cite{emerton_helm_families} compatible with (a modified version due to Breuil--Schneider of) the classical one and the one constructed by Vignéras.

In terms of the theta correspondence and the Weil representation, a generalisation to $\ell$-modular representation theory has already been considered in the thesis of Alberto M\'inguez \cite{minguez_thesis}. Taking an \textit{ad-hoc} analogue of the Weil representation for type II dual pairs, he proves that a bijective correspondence holds when the characteristic is big enough compared to the size of the dual pair at play. In order to develop a modular theory of the theta correspondence, this analogue is not sufficient and one needs a proper construction of the Weil representation for coefficient fields, or even coefficient rings. 

In~\cite{shin_weil}, Sug Woo Shin achieves this for coefficient rings such that the associated affine scheme is locally noetherian, by the use of geometric methods such as a Stone-von Neumann theorem involving abelian schemes. In \cite{ct}, the authors build a Weil representation with coefficients in integral domains following the original approach of Weil. The other representation-theoretic strategy, using a non-geometric version of the Stone-von Neumann theorem, has been carried out in \cite{trias_theta1}. The latter allows one to recover most of the classical objects and study them in detail, such as the metaplectic group, the metaplectic cocycle, and the lifts of dual pairs. Furthermore, this approach generalises in a nice way in families without the need of particular assumptions on the coefficient rings, improving the first two mentioned papers whose hypotheses (locally noetherian affine scheme, or, integral domains) turn out to be more restrictive. 

The present paper brings a broader construction of the Weil representation with coefficients in any $\mathcal{A}$-algebra, where $\mathcal{A}$ is a minimal ring specified below. In addition, exhibiting a minimal Weil representation, called ``universal'' below, does not appear in any previous work; nor the focus on extending scalars. The rest of the introduction is split into two parts: in the first half we give more detail about the results we obtain along these lines, as well as considerations about the metaplectic group and the metaplectic cocycle; in the second half, we explain how we expect to use this to study an integral theta correspondence, with particular focus on the special case of $(\textup{GL}_1,\textup{GL}_1)$.

Let $F$ be either a local non-archimedean field, or a finite field, of residual characteristic $p$ and residual cardinality $q$, but of characteristic not $2$. The minimal condition mentioned above amounts to requiring two things: first that a non-trivial smooth additive character $\psi$ of $F$ exists, allowing Fourier transform techniques; second that $p$ is invertible, that is a condition in terms of Haar measures. Write $\mathcal{K} = \mathbb{Q}[\zeta_p]$ when $F$ has positive characteristic, and $\mathcal{K}=\mathbb{Q}[\zeta_{p^\infty}]$ when $F$ has characteristic $0$. The minimal ring $\mathcal{A}$ satisfying the previous two conditions is the integral closure of $\mathbb{Z}[\frac{1}{p}]$ in $\mathcal{K}$.

Fix from now on a non-trivial smooth character $\psi : F \to \mathcal{A}^\times$. Notations will be simplified in the introduction to be lighter than that in the main body of the paper. For any $\mathcal{A}$-algebra $\mathcal{B}$ with structure morphism $\varphi$, the character $\psi^\mathcal{B} = \varphi \circ \psi$ is a non-trivial character of $F$ with values in $\mathcal{B}$. More generally, if~$\chi$ is a character of any group with values in~$\mathcal{A}$, we write~$\chi^{\mathcal{B}}=\varphi\circ\chi$.

\subsection{Theory of the Weil representation over an $\mathcal{A}$-algebra $\mathcal{B}$}

The theory developed in \cite[Chap. 2]{mvw} for complex representations and in \cite{trias_theta1} for $\ell$-modular representations finds a natural generalisation for an $\mathcal{A}$-algebra $\mathcal{B}$. Note that there are no restrictive assumptions on the $\mathcal{A}$-algebra considered. In particular, it is not necessarily an integral domain.

\subsubsection*{Stone-von Neumann over $\mathcal{A}$-algebras} 

Let~$A$ be a self-dual subgroup in the symplectic space $W$ and $\psi_A$ a character of the group $A_H=A\times F$ extending $\psi$. The theorem below gathers together in a succinct way the main results from Sections~\ref{models_associated_to_self_dual_subgp_section} \& \ref{changing_models_section}. It is the core part of the classical Stone-von Neumann theorem when $\mathcal{B}=\mathbb{C}$ \cite[Chap. 2, Th. I.2]{mvw} and its generalisation when $\mathcal{B}$ is a field of characteristic different from $p$ \cite[Th. 2.1]{trias_theta1}. Working over a general ring, the notion of ``irreducible representation'' is too restrictive. Instead, we say that a $\mathcal{B}[G]$-module $V$ is \emph{everywhere irreducible} if the representation $V \otimes_\mathcal{B} k(\mathcal{P})$ is irreducible for all~$\mathcal{P} \in \textup{Spec}(\mathcal{B})$, where~$k(\mathcal P)$ is the field of fractions of~$\mathcal B/\mathcal P$. 

\begin{theorem} \label{metaplectic_reps_first_thm} 
Set $V_A^\mathcal{B} = \textup{ind}_{A_H}^H(\psi_A^\mathcal{B}) \in \textup{Rep}_\mathcal{B}(H)$.
\begin{enumerate}[label=\textup{\alph*)}]
\item \label{everywhere_irreducible_point} Set $V_A^\mathcal{B}$ is everywhere irreducible, and is admissible.
\item \label{extension_scalar_point} We have $V_A^{\mathcal{A}} \otimes_\mathcal{A} \mathcal{B} = V_A^\mathcal{B}$.
\item \label{changing_models_point} For~$A'$ any self-dual subgroup in~$W$ and~$\psi_{A'}$ an extension of~$\psi$ to~$A'_H$, one has:
$$\textup{Hom}_{\mathcal{B}[H]}(V_A^\mathcal{B}, V_{A'}^\mathcal{B}) \simeq \mathcal{B}.$$
\end{enumerate} \end{theorem}

A consequence of \ref{everywhere_irreducible_point} and \ref{changing_models_point} is that the isomorphism class of the representation $V_A^\mathcal{B}$ does not depend on the choices of $A$ and $\psi_A$. When~$\mathcal{B}$ is a field, this representation is also irreducible.

The full Stone-von Neumann Theorem for fields~$\mathcal{B}$ also asserts that any irreducible $V \in \textup{Rep}_\mathcal{B}(H)$, such that $V|_F$ is $\psi^\mathcal{B}$-isotypic, is in the isomorphism class defined by $V_A^\mathcal{B}$. We do not pursue such a precise result over rings. However, for most of the applications using Stone-von Neumann, and the Weil representation, one usually sticks to the explicit models given by the representations $V_A^\mathcal{B}$, where $A$ is a self-dual lattice or a lagrangian, so our result is sufficient.

\subsubsection*{Weil representations over $\mathcal{A}$-algebras}

Let $A$ be a self-dual subgroup in $W$. According to Section \ref{weil_over_A_algebra_section}, the action of $\textup{Sp}(W)$ on $H$ induces a projective representation 
$\sigma_\mathcal{B} : \textup{Sp}(W) \to \textup{PGL}_\mathcal{B}(V_A^\mathcal{B})$. 
Denote by $\textsc{red} : \textup{GL}_\mathcal{B}(V_A^\mathcal{B}) \to \textup{PGL}_\mathcal{B}(V_A^\mathcal{B})$ the quotient morphism. To lift a projective representation, one uses the fiber product construction to obtain a representation of some central extension. Looking at the fiber product of $\sigma_\mathcal{B}$ and $\textsc{red}$ above $\textup{PGL}_\mathcal{B}(V_A^\mathcal{B})$, Proposition \ref{weil_rep_def_prop} defines:
$$\xymatrix{
		\widetilde{\textup{Sp}}_{\psi,A}^\mathcal{B}(W) \ar[r]^{\omega_{\psi,A}^\mathcal{B}} \ar[d]_{p_\mathcal{B}} & \textup{GL}_\mathcal{B}(V_A^\mathcal{B})  \ar[d]^{\textsc{red}} \\
		\textup{Sp}(W) \ar[r]^{\sigma_\mathcal{B}} & \textup{PGL}_\mathcal{B}(V_A^\mathcal{B})}.$$

\begin{defi_sans_num} The Weil representation associated to $\psi$ and $A$ with coefficients in $\mathcal{B}$ is the representation $(\omega_{\psi,A}^\mathcal{B},V_A^\mathcal{B})$ of the central extension $\widetilde{\textup{Sp}}_{\psi,A}^\mathcal{B}(W)$ of $\textup{Sp}(W)$ by $\mathcal{B}^\times$. \end{defi_sans_num}

Recalling the canonical identification $V_A^\mathcal{A} \otimes_\mathcal{A} \mathcal{B} = V_A^\mathcal{B}$ from \ref{extension_scalar_point} of Theorem~\ref{metaplectic_reps_first_thm} above, our Theorem \ref{weil_rep_compatib_PhiB_thm} ensures the compatibility:

\begin{theorem} \label{weil_rep_comptib_PhiB_thm_intro} There exists a canonical morphism of central extensions:
$$\widetilde{\phi_\mathcal{B}} : \widetilde{\textup{Sp}}_{\psi,A}^\mathcal{A}(W) \to \widetilde{\textup{Sp}}_{\psi,A}^\mathcal{B}(W)$$
whose image is a central extension of $\textup{Sp}(W)$ by $\varphi(\mathcal{A})^\times$. Moreover, there is a commuting diagram:
$$\xymatrix{
		\widetilde{\textup{Sp}}_{\psi,A}^\mathcal{A}(W) \ar[r]^{\omega_{\psi,A}^\mathcal{A}} \ar[d]^{\widetilde{\phi_\mathcal{B}}} & \textup{GL}_\mathcal{A}(V_A)  \ar[d] \\
		\widetilde{\textup{Sp}}_{\psi,A}^\mathcal{B}(W) \ar[r]^{\omega_{\psi,A}^\mathcal{B}} & \textup{GL}_\mathcal{B}(V_A^\mathcal{B}).
		}$$ \end{theorem}

Moreover there exist canonical identifications between these central extensions as $A$ varies: for any other self-dual subgroup $A'$, Corollary \ref{canonical_identif_central_ext_B_cor} defines a canonical morphism of central extension such that $\omega_{\psi,A}^\mathcal{B}$ and $\omega_{\psi,A'}^\mathcal{B}$ agree, where the term ``agree'' is made precise in the corollary mentioned. So the Weil representation $\omega_\psi^\mathcal{B}$ associated to $\psi$ is well-defined in the sense that the isomorphism class of $\omega_{\psi,A}^\mathcal{B}$ does not depend on $A$.

\subsubsection*{The metaplectic group over $\mathcal{A}$-algebras}

The isomorphism class of $\widetilde{\textup{Sp}}_{\psi,A}^\mathcal{B}(W)$, as a central extension of $\textup{Sp}(W)$ by $\mathcal{B}^\times$, does not depend on the choice of $A$ or $\psi_A$. In addition, the canonical isomorphism of central extensions induced by $V_A^\mathcal{B} \simeq V_{A'}^\mathcal{B}$ is compatible with the fiber product projections. Therefore one can speak of \textit{the} metaplectic group over $\mathcal{B}$ associated to $\psi$ as any element in the previous isomorphism class. Even if these groups may be isomorphic for different $\psi$, there does not necessarily exist an isomorphism compatible with the fiber product construction: in this sense these groups 
\shaun{do} %
depend on $\psi$.

We endow the module $V_A^\mathcal{B}$ with the discrete topology and the group $\textup{GL}_\mathcal{B}(V_A^\mathcal{B})$ with the compact-open one. Then Corollary \ref{met_gp_over_A_over_C_cor} compares the situation over $\mathcal{A}$ with that for the classical metaplectic group. Indeed if we endow $\mathbb{C}$ with a structure of $\mathcal{A}$-algebra, then:

\begin{prop_sans_num} The group $\widetilde{\textup{Sp}}_{\psi,A}^\mathcal{A}(W)$ is an open topological subgroup of $\widetilde{\textup{Sp}}_{\psi,A}^\mathbb{C}(W)$.  \end{prop_sans_num}

Here the natural topology on $\widetilde{\textup{Sp}}_{\psi,A}^\mathbb{C}(W)$ is that as a subgroup of $\textup{Sp}(W) \times \textup{GL}_\mathbb{C}(V_A^\mathbb{C})$. The classical metaplectic group is known to be locally profinite, and so is the metpalectic group over $\mathcal{A}$ because of the proposition. Define now the derived group:
$$\widehat{\textup{Sp}}_{\psi,A}^\mathcal{B}(W) =[\widetilde{\textup{Sp}}_{\psi,A}^\mathcal{B}(W),\widetilde{\textup{Sp}}_{\psi,A}^\mathcal{B}(W)].$$
When $\mathcal{B}=\mathbb{C}$, this derived group is the reduced metaplectic group when $F$ is local non-achimedean, or the symplectic group when $F$ is finite, except in the exceptional case $F=\mathbb{F}_3$ and $\textup{dim}_F(W) =2$. Acccording to Proposition \ref{metaplectic_group_over_A_prop}, one has a canonical isomorphism of central extensions:
$$\widehat{\textup{Sp}}_{\psi,A}^\mathcal{A}(W) \simeq \widehat{\textup{Sp}}_{\psi,A}^\mathbb{C}(W).$$
Proposition \ref{metaplectic_group_over_B_prop} sheds light on the structure of the metaplectic group:

\begin{theorem}  One has the following properties:
\begin{enumerate}[label=\textup{\alph*)}]
\item the group $\widetilde{\textup{Sp}}_{\psi,A}^\mathcal{B}(W)$ is the fiber product in the category of topological groups of the morphisms $\sigma_\mathcal{B}$ and $\textup{\textsc{red}}$, having the subspace topology in $\textup{Sp}(W) \times \textup{GL}_\mathcal{B}(V_A^\mathcal{B})$;
\item the representation $\omega_{\psi,A}^\mathcal{B} : \widetilde{\textup{Sp}}_{\psi,A}^\mathcal{B}(W)  \to \textup{GL}_\mathcal{B}(V_A^\mathcal{B})$ is smooth;
\item the map $\widetilde{\phi_\mathcal{B}}$ of Theorem~\ref{weil_rep_comptib_PhiB_thm_intro} is open and continuous, and $\widetilde{\textup{Sp}}_{\psi,A}^\mathcal{B}(W)$ is locally profinite;
\item \label{isomorphism_thm_met_gp_B_point} considering derived groups, the map $\widetilde{\phi_\mathcal{B}}$ restricts to:
\begin{enumerate} \item[i)] a surjection $\widehat{\textup{Sp}}_{\psi,A}^\mathcal{A}(W) \to \widehat{\textup{Sp}}_{\psi,A}^\mathcal{B}(W)$ with kernel $\{\pm 1 \}$ and image isomorphic to $\textup{Sp}(W)$ if $F$ is local non-archimedean and $\textup{char}(\mathcal{B})=2$;
\item[ii)] an isomorphism $\widehat{\textup{Sp}}_{\psi,A}^\mathcal{A}(W) \simeq \widehat{\textup{Sp}}_{\psi,A}^\mathcal{B}(W)$ otherwise. \end{enumerate} \end{enumerate} \end{theorem}

Again exclude the exceptional case, which is considered in the separate Remark \ref{exceptional_case_sigmaA_sigmaB_rem}. In Section \ref{reduced_cocycle_over_B_section}, we prove:

\begin{theorem} There exists a section $\varsigma^\mathcal{B} : \textup{Sp}(W) \to \widehat{\textup{Sp}}_{\psi,A}^\mathcal{B}(W)$ compatible with that defined over $\mathcal{A}$ and such that the associated $2$-cocyle $\hat{c}_\mathcal{B}$ has image:
\begin{itemize}
\item $\{ 1 \}$ if $F$ is finite or $\textup{char}(\mathcal{B})=2$;
\item $\{ \pm 1\}$ if $F$ is local non-archimedean and $\textup{char}(\mathcal{B}) \neq 2$.
\end{itemize}
\end{theorem}

\subsubsection*{Families of Weil representations}

The consequence of these results is that one may speak of a universal Weil module $\omega_\psi^\mathcal{A}$ over $\mathcal{A}$ associated to $\psi$: that is (see Proposition~\ref{weil_rep_A_to_B_prop}) any Weil representation $\omega_\psi^\mathcal{B}$ with coefficients in $\mathcal{B}$ arises from the scalar extension of this universal Weil module. Thus, according to the compatibility in Theorem \ref{weil_rep_comptib_PhiB_thm_intro}, the Weil representation $\omega_\psi^\mathcal{A}$ is a family of Weil representations over the residue fields of $\textup{Spec}(\mathcal{A})$.

\subsection{Towards an integral theta correspondence}

In the rest of the introduction, we give some new ideas and speculate in the direction of an integral theta correspondence. As an illustration, we study in detail the case of the type II dual pair $(F^\times,F^\times)$ but it is only this example which is part of the main body of the paper. Thus the text below is a kind of story about the bigger picture to motivate our study and can be omitted if the reader is only interested in the Weil representation itself.

Suppose again $F$ is local non-archimedean. For 
general dual pairs $(H_1,H_2)$, one usually considers the Weil representation with coefficients in a field, along with its biggest $\pi_1$-isotypic quotients for $\pi_1$ running through the irreducible representations of~$H_1$. However, there is no  natural definition of what a good biggest isotypic quotient over a ring should be. But there is another approach with a coarser invariant in terms of the Bernstein centre, giving a bigger representation. In order to lighten notations further, we omit the reference to $\psi$ from now on.

\subsubsection*{Replacing biggest isotypic quotients: a heuristic approach}

Suppose in this paragraph that $\mathcal{B}$ is an algebraically closed field. Let $\mathfrak{z}_\mathcal{B}(H_1)$ be the Bernstein centre of $H_1$. A character of the Bernstein of the centre is a $\mathcal{B}$-algebra morphism $\eta_1 : \mathfrak{z}_\mathcal{B}(H_1) \to \mathcal{B}$. The set of such characters correspond bijectively to the points in $\textup{Specmax}(\mathfrak{z}_\mathcal{B}(H_1))$. Denote by
$\eta_{\pi_1} : \mathfrak{z}_\mathcal{B}(H_1) \to \mathcal{B}$ the character associated to~$\pi_1$. The construction of the biggest $\pi_1$-isotypic quotient factors through the biggest $\eta_{\pi_1}$-isotypic quotient, in the sense that for any $V \in \textup{Rep}_\mathcal{B}(H_1)$, the quotient $V \to V_{\pi_1}$ factors through $V \to V \otimes_{\mathfrak{z}_\mathcal{B}(H_1)} \eta_{\pi_1}$. Denote by $V_{\eta_{\pi_1}}$ the latter representation.  Regardless of the characteristic of $\mathcal{B}$, and similarly to $V_{\pi_1} \in \textup{Rep}_\mathcal{B}(H_1 \times H_2)$ when $V \in \textup{Rep}_\mathcal{B}(H_1 \times H_2)$, one has $V_{\eta_{\pi_1}} \in \textup{Rep}_\mathcal{B}(H_1\times H_2)$.

When the characteristic $\ell$ of $\mathcal{B}$ is banal with respect to $H_1$, that is when $\ell$ does not divide the pro-order $|H_1|$ of $H_1$, the set of characters of $\mathfrak{z}_\mathcal{B}(H_1)$ is in bijection with the set of cuspidal supports in $\textup{Rep}_\mathcal{B}(H_1)$ and  we expect the following to hold for all $\eta_1$ in a Zariski open subset of $\textup{Specmax}(\mathfrak{z}^\mathcal{B}(H_1))$:
$$V_{\eta_1} \simeq \bigoplus_{\pi_1 \in \eta_1} V_{\pi_1}$$
where $\pi_1 \in \eta_1$ means $\eta_{\pi_1}=\eta_1$, that is $\pi_1$ has cuspidal support corresponding to $\eta_1$.

Outside the banal setting, it seems risky to state any precise results. Already some key facts fail: the maximal ideals of $\mathfrak{z}_\mathcal{B}(H_1)$ are no longer in bijection with cuspidal supports. However the biggest $\pi_1$-isotypic quotient $V_{\pi_1}$ always is a quotient of the bigger representation~$V_{\eta_{\pi_1}}$, so this last construction encapsulates more information. In addition, we expect this object to behave in a nicer way for coefficient rings as it keeps track of congruences.

\subsubsection*{Illustration for the type II dual pair $(F^\times,F^\times)$}

The category $\textup{Rep}_\mathcal{B}(F^\times)$ can be decomposed according to the level and we denote by $\textup{Rep}_\mathcal{B}^0(F^\times)$ the level $0$ direct factor category. This category is Morita-equivalent to the category of $\mathfrak{z}^0$-modules, where $\mathfrak{z}^0$ is the commutative ring $\mathcal{B}[F^\times / 1 + \varpi_F \mathcal{O}_F]$. Up to choosing a uniformiser $\varpi_F$ and a primitive $(q-1)$-th root of unity $\zeta$ in $F^\times$, this ring is isomorphic to $\mathcal{B}[X^{\pm 1},Z]/(Z^{q-1}-1)$ by sending $X$ to $\varpi_F$ and $Z$ to $\zeta$. Instead of considering biggest isotypic quotients associated to irreducible representations in $\textup{Rep}_\mathcal{B}(F^\times)$, Section \ref{specialisation_using_the_center_section} considers more general isotypic families of representations using the explicit description of (the center of) $\mathfrak{z}^0$. 

\begin{defi_sans_num} Let $V \in \textup{Rep}_\mathcal{B}(F^\times)$. When $\mathcal{C}$ is a commutative $\mathcal{B}$-algebra and $\eta : \mathfrak{z}^0 \to \mathcal{C}$ is a $\mathcal{B}$-algebra morphism, the repesentation $V_\eta = V \otimes_{\mathfrak{z}^0} \eta \in \textup{Rep}_\mathcal{C}(F^\times)$ may be thought as the ``biggest $\eta$-isotypic quotient of $V$.'' \end{defi_sans_num}

\begin{rem_sans_num} Unlike the situation of the biggest isotypic quotient, $V$ does not necessarily surject onto $V_\eta$ if $\eta$ is not surjective. So in general $V_\eta$ is not a quotient of $V$, but the image of $V$ in $V_\eta$ generates $V_\eta$ as a $\mathcal{C}$-module. \end{rem_sans_num}

When $\mathcal{B}'$ is a $\mathcal{B}$-algebra, denote by $(\textbf{1}_{F^\times}^{\mathcal{B}'},\mathcal{B}') \in \textup{Rep}_\mathcal{B}^0(F^\times)$ the trivial $\mathfrak{z}^0$-module isomorphic to $\mathcal{B}'$. Denote by $(\chi_\mathcal{B},\mathcal{B})$ the character with $\chi_\mathcal{B}(\varpi_F) = q \in \mathcal{B}^\times$ and $\chi_\mathcal{B}|_{\mathcal{O}_F^\times} = 1_\mathcal{B}$. Thus $\chi_\mathcal{B}$ is the inverse of the norm $| \cdot |_F$.

Let $I_{\textbf{\textup{1}}}$ be the ideal in $\mathfrak{z}^0$ corresponding to $(X-1,Z-1)$ in $\mathcal{B}[X^{\pm 1},Z]/(Z^{q-1}-1)$. Denote the quotient map $\eta_{\textbf{\textup{1}}} : \mathfrak{z}^0 \to \mathfrak{z}^0 / I_{\textbf{\textup{1}}}$. Consider the isotypic family $V_{\eta_{\textbf{\textup{1}}}}$ associated to $\eta_{\textbf{\textup{1}}}$ with respect to the action of the first copy of $F^\times$ on $V$. Take the same convention for $I$ corresponding to $(X-q,Z-1)$ with $\eta$ being the quotient map.

\begin{theorem} \label{isotypic_families_GL1_GL1_thm} One has in $\textup{Rep}_\mathcal{B}(F^\times)$ the following isomorphisms:

\begin{enumerate}[label=\textup{\alph*)}]
\item $V_{\eta_\textbf{\textup{1}}} \simeq \textbf{\textup{1}}_{F^\times}^{\mathcal{B}/(q-1)\mathcal{B}} \oplus \textbf{\textup{1}}_{F^\times}^\mathcal{B}$;
\item $V_\eta \simeq \textbf{\textup{1}}_{F^\times}^{\mathcal{B}/(q-1) \mathcal{B}} \oplus \chi_\mathcal{B}.$ \end{enumerate} \end{theorem}

The subrepresentation $\textbf{\textup{1}}_{F^\times}^{\mathcal{B}/(q-1) \mathcal{B}}$ is in a certain sense the ``defect'' in the theta correspondence. This is a pure $(q-1)$-torsion submodule, whereas the other part is a free $\mathcal{B}$-module of rank $1$. When $\mathcal{B}$ is a field, this defect vanishes if and only if the characteristic $\ell$ of $\mathcal{B}$ does not divide $q-1$, that is $\ell$ is banal with respect to $F^\times$.

\subsubsection*{Further example}

Using this interpretation in terms of 
characters of the Bernstein centre seems to be more suitable when $\mathcal{B}$ is a ring. Indeed recall the situation in \cite[Sec. 5.3]{trias_theta1} where $F$ has odd residual characteristic and $(H_1,H_2)$ is a type I dual pair that is split in the metaplectic group. Let $\ell$ be a prime that does not divide the pro-order of $H_1$ and endow $\mathcal{B}=W(\overline{\mathbb{F}_\ell})$ with an $\mathcal{A}$-algebra stucture. Let $K$ be the fraction field of $\mathcal{B}$. For any absolutely irreducible cuspidal $\Pi_1 \in \textup{Rep}_K(H_1)$, one has the equality $V_{\Pi_1} = V_{\eta_{\Pi_1}}$ for $V \in \textup{Rep}_K(H_1)$.

The reduction modulo $\ell$ of $\Pi_1$ is obtained by choosing a sable lattice $L_{\Pi_1}$ in $\Pi_1$. The reduction modulo $\ell$ of this lattice is an irreducible representation $\pi_1$ whose isomorphism class does not depend on the choice of $L_{\Pi_1}$. We refer to \cite[Sec. 5.3]{trias_theta1} for more details, but what is important here is that similarly to $\Pi_1$, we have $V_{\pi_1} = V_{\eta_{\pi_1}}$ for $V \in \textup{Rep}_{\overline{\mathbb{F}_\ell}}(H_1)$. Actually this comes along with some compatibilities to scalar extension. Indeed there exists a character $\eta_1 : \mathfrak{z}_{W(\overline{\mathbb{F}_\ell})}(H_1) \to W(\overline{\mathbb{F}_\ell})$ of the integral Bernstein centre such that $\eta_1 \otimes_{W(\overline{\mathbb{F}_\ell})} \overline{\mathbb{F}_\ell} = \eta_{\pi_1}$ and $\eta_1 \otimes_{W(\overline{\mathbb{F}_\ell})} K = \eta_{\Pi_1}$. This yields, for any $V \in \textup{Rep}_{W(\overline{\mathbb{F}_\ell})}(H_1 \times H_2)$, the following canonical morphisms in $\textup{Rep}_{W(\overline{\mathbb{F}_\ell})}(H_1 \times H_2)$:

$$\xymatrix{
V_{\eta_1} \ar@{->>}[d] \ar[r] &  V_{\eta_1} \otimes_{W(\overline{\mathbb{F}_\ell})} K = (V \otimes_{W(\overline{\mathbb{F}_\ell})} K)_{\eta_{\Pi_1}} \\
V_{\eta_1}  \otimes_{W(\overline{\mathbb{F}_\ell})} \overline{\mathbb{F}_\ell} = (V \otimes_{W(\overline{\mathbb{F}_\ell})} \overline{\mathbb{F}_\ell})_{\eta_{\pi_1}} &  
}.$$
When $V=\omega$ is the Weil representation with coefficients in $W(\overline{\mathbb{F}_\ell})$, the Weil representations with coefficients in the residue fields $\overline{\mathbb{F}_\ell}$ and $K$ of $W(\overline{\mathbb{F}_\ell})$ are $\bar{\omega} = \omega \otimes_{W(\overline{\mathbb{F}_\ell})} \overline{\mathbb{F}_\ell}$ and $\Omega = \omega \otimes_{W(\overline{\mathbb{F}_\ell})} K$ respectively. The biggest isotypic quotients are $\Omega_{\eta_{\Pi_1}} \simeq \Pi_1 \otimes_K \Theta (\Pi_1)$ and $\bar{\omega}_{\eta_{\pi_1}} \simeq \pi_1 \otimes_{\overline{\mathbb{F}_\ell}} \Theta(\pi_1)$, where $\Theta(\Pi_1) \in \textup{Rep}_K(H_2)$ and $\Theta(\pi_1) \in \textup{Rep}_{\overline{\mathbb{F}_\ell}}(H_2)$. So $\omega_{\eta_1}$  is a good family object because its generic fiber is $\Pi_1 \otimes_K \Theta(\Pi_1)$ and its special fiber is $\pi_1 \otimes_{\overline{\mathbb{F}_\ell}} \Theta(\pi_1)$. In addition $\Theta(\Pi_1)$ is irreducible, when it is non-zero and $\omega_{\eta_1}$ is a $W(\overline{\mathbb{F}_\ell})[H_1 \times H_2]$-lattice in $\Pi_1 \otimes_K \Theta(\Pi_1)$. Furthermore, when $\ell$ is banal with respect to $H_2$ and $\Theta(\Pi_1)$ is cuspidal, the representation $\Theta(\pi_1)$ is the reduction modulo $\ell$ of $\Theta(\Pi_1)$ and is therefore irreducible \cite[Th. 5.17]{trias_theta1}. To relate $\Theta(\Pi_1)$ and $\Theta(\pi_1)$ in general, one needs to explicitly know which lattice in $\Pi_1 \otimes_K \Theta(\Pi_1)$ is $\omega_{\eta_1}$.

\subsubsection*{First expectations}

Of course in the most general situation, \textit{i.e.} when the coefficient ring $\mathcal{B}$ is $\mathbb{Z}[\frac{1}{p}]$ (or $\mathcal{A}$), exhibiting blocks, as well as their centres, is a daydream. However, one can play with:
\begin{itemize}
\item ``simpler'' rings $\mathcal{B}$ (fields, local rings, banal characteristic, etc.);
\item special classes of representations (cuspidals, level 0, etc.);
\item easier groups in the dual pair (small dimension, general linear, etc.).
\end{itemize}
As recalled, this has been achieved in \cite[Sec. 5.3]{trias_theta1} for type I dual pairs $(H_1,H_2)$ over the local ring $W(\overline{\mathbb{F}_\ell})$ when $\ell$ is banal with respect to $H_1$, looking at the block defined by a (super)cuspidal representation. In Section \ref{features_GL1_GL1_section}, we consider the (very simple) pair $(F^\times,F^\times)$, especially for level $0$ representations. For bigger type II dual pairs $(\textup{GL}_n(F),\textup{GL}_m(F))$ and coefficients rings being made of Witt vectors, the work \cite{helm_bernstein_center} seems to be the cornerstone to tackle the problem. Based on calculations in small dimensions, we make the following two conjectures.

\paragraph{Torsion principle.} When the pro-order of $H_1$, or that of~$H_2$, is
not invertible in $\mathcal{B}$, we expect the failure of the theta correspondence to appear as some $|H_1|_f |H_2|_f$-torsion submodule in the family object, where~$|H_i|_f$ denotes the prime-to-$p$ part of the pro-order of~$H_i$. Thanks to Theorem \ref{isotypic_families_GL1_GL1_thm}, this principle is made a bit more precise when $(H_1,H_2)=(F^\times,F^\times)$.

\paragraph{Bijection principle for characters of the Bernstein centre.} Another problem is the following. When $\eta_1 : \mathfrak{z}_\mathcal{B}(H_1) \to \mathcal{B}$ is a character, are there any nice properties of $(\omega^\mathcal{B})_{\eta_1}$, where $\omega^\mathcal{B}$ is the Weil representation over $\mathcal{B}$? For instance, it seems that the action of $\mathfrak{z}_\mathcal{B}(H_2)$ can also be described in terms of a character of $\mathfrak{z}_\mathcal{B}(H_2)$. Indeed one expects that there exists a character $\eta_2 : \mathfrak{z}_\mathcal{B}(H_2) \to \mathcal{B}$ such that $((\omega^\mathcal{B})_{\eta_1})_{\eta_2} = (\omega^\mathcal{B})_{\eta_1}$. Denoting by $\eta_1 \otimes_\mathcal{B} \eta_2$ the natural character $\mathfrak{z}_\mathcal{B}(H_1) \otimes_\mathcal{B} \mathfrak{z}_\mathcal{B}(H_2) \to \mathcal{B}$, we expect even more: $(\omega^\mathcal{B})_{\eta_1} = (\omega^\mathcal{B})_{\eta_2} = (\omega^\mathcal{B})_{\eta_1 \otimes_\mathcal{B} \eta_2}$. Writing $\eta_2 = \theta(\eta_1)$, one could then speak of a theta correspondence in terms of characters of the respective Bernstein centres because $\theta$ would induce a bijection:
$$\{ \eta_1 : \mathfrak{z}_\mathcal{B}(H_1) \to \mathcal{B} \ | \ (\omega^\mathcal{B})_{\eta_1} \neq 0 \} \overset{\theta}{\simeq} \{ \eta_2 : \mathfrak{z}_\mathcal{B}(H_2) \to \mathcal{B} \ | \ (\omega^\mathcal{B})_{\eta_2} \neq 0 \}.$$

\noindent \textbf{Acknowledgements:} I would like to thank Shaun Stevens for his useful comments, as well as Gil Moss for fruitful discussions.

\tableofcontents

\section{Preliminaries}

\subsection{Notations} \label{notations_section}

All along the paper $F$ will be a field of characteristic not 2, which is either finite or local non archimedean. The residual characteristic and cardinality of $F$ are denoted as usual $p$ and $q$. To turn $F$ into a topological field one considers the usual locally profinite topology. One of the many equivalent formulations of the latter is ``locally compact and totally disconnected''.

\paragraph{$\mathcal{K}$ and $\mathcal{A}$.} Let $\mathcal{K}$ be the field defined in the following two cases:
\begin{itemize}[$\bullet$]
\item $\mathcal{K}$ is the cyclotomic extension of $\mathbb{Q}$ containing the $p^{\textup{th}}$ roots of unity, when the characateristic of $F$ is positive;
\item $\mathcal{K}$ is the algebraic extension of $\mathbb{Q}$ containing all the $p$ power roots of unity, when the characteristic of $F$ is zero.
\end{itemize}
One can write $\mathcal{K} = \mathbb{Q}(\zeta_p)$ by fixing a generator $\zeta_p$ in the first case; in the second however, no generator exists, though the notation $\mathcal{K} = \mathbb{Q}(\zeta_{p^\infty})$ is commonly used. Based on classical results for cyclotomic extensions, the integral closure $\mathcal{O}_\mathcal{K}$ of $\mathbb{Z}$ in $\mathcal{K}$ is, in the first case $\mathbb{Z}[\zeta_p]$, and in the second $\mathbb{Z}[\zeta_{p^\infty}]$. From now on, let $\mathcal{A}$ be the subring of $\mathcal{K}$ obtained from the ring of integers $\mathcal{O}_\mathcal{K}$ by inverting $p$, that is:
$$\mathcal{A} = \mathcal{O}_\mathcal{K}\left[ \frac{1}{p} \right].$$

\paragraph{$\mathcal{A}$-algebras.} By convention, the term $\mathcal{A}$-algebra will refer to commutative rings $\mathcal{B}$ endowed with an $\mathcal{A}$-algebra structure, that is, a ring morphism $\varphi : \mathcal{A} \to \mathcal{B}$. In order to avoid confusion, those $\mathcal{B}$ considered always are unitary rings and $\varphi$ maps the neutral multiplicative element of $\mathcal{A}$ to that of $\mathcal{B}$. Denote $\textup{char}(\mathcal{B})$ the characteristic of $\mathcal{B}$, that is the natural number such that $\{ k \in \mathbb{Z} \ | \ \varphi(k)=0\} = \textup{char}(\mathcal{B}) \mathbb{Z}$. The ring morphism $\varphi$ induces a group morphism $\mathcal{A}^\times \to \mathcal{B}^\times$ between the group of units of $\mathcal{A}$ and that of $\mathcal{B}$. Denote $\mu^p(\mathcal{B}) = \{ \zeta \in \mathcal{B}^\times \ | \ \exists k \in \mathbb{Z}, \zeta^{p^k}=1 \}$ for the group of elements in $\mathcal{B}^\times$ having order a power of $p$.

\paragraph{Character $\psi^\mathcal{B}$.} Let $\mathcal{B}$ be an $\mathcal{A}$-algebra.  Then $\varphi$ restricts injectively to the group of roots in $\mathcal{A}^\times$ having order a power of $p$, that is the group morphism $\varphi : \mu^p(\mathcal{A}) \to \mu^p(\mathcal{B})$ is injective. Indeed, given two distinct roots of unity $\zeta$ and $\zeta'$ in $\mu^p(\mathcal{A})$, their difference $\zeta - \zeta'$ is in $\mathcal{A}^\times$ because $p \in \mathcal{A}^\times$, so they define two distinct elements in $\varphi(\mathcal{A}) = \mathcal{A} / \textup{Ker}(\varphi)$. Therefore one can build, out of any non-trivial smooth character $\psi : F \to \mathcal{A}^\times$, a character $\varphi \circ \psi : F \to \mathcal{B}^\times$ which is still non-trivial and smooth. In order to keep track of the ring considered, one uses a superscript to refer to the $\mathcal{A}$-algbera at stake. From now on, fix such a non-trivial smooth $\psi^\mathcal{A} : F \to \mathcal{A}^\times$ and set:
$$\psi^\mathcal{B}=\varphi \circ \psi^\mathcal{A}.$$

\paragraph{Smooth representations.} Let $G$ be a locally profinite group. Let $R$ be a commutative unitary ring. An $R[G]$-module $V$ is said to be smooth if for all $v \in V$, the stabiliser $G_v$ of $v$ is open in $G$. One denotes $\textup{Rep}_R(G)$ the category of smooth $R[G]$-modules. For any closed subgroup $H$ in $G$, the induction functor $\textup{Ind}_H^G$ associates to any $(\sigma,W) \in \textup{Rep}_R(H)$, the representation $\textup{Ind}_H^G(W) \in \textup{Rep}_R(G)$ of locally constant functions on $G$ taking values in $W$ and satisfying $f(h g) = \sigma(h) \cdot f(g)$ for all $g \in G$ and $h \in H$. The compact induction $\textup{ind}_H^G$ is the subfunctor of $\textup{Ind}_H^G$ made of functions compactly supported modulo $H$, that is the subspace of functions $f \in \textup{Ind}_H^G(W)$ such that the image of $\textup{supp}(f)$ in $H \backslash G$ is a compact set. A representation $V \in \textup{Rep}_R(G)$ is said to be admissible if for all compact open subgroups $K$ in $G$, the set of $K$-invariants $V^K = \{ v \in V \ | \ g \cdot v = v \}$ is finitely generated as an $R$-module.

\paragraph{Haar measures.} Let $G$ be a locally profinite group. In the following, we use the notations of \cite[I.1 \& I.2]{vig_book}. The pro-order $|G|$ of $G$ is the least common multiple, in the sense of supernatural integers, of the orders of its open compact subgroups. To be more explicit, $|G|$ is a function $\mathcal{P} \to \mathbb{N} \cup {\infty}$ on the set of prime numbers $\mathcal{P}$. This decomposes in an obvious way into two parts having disjoint supports, namely the finite part $|G|_f$ and the infinite one $|G|_\infty$. The only situation occuring in the present work is $|G| = |G|_f \times |G|_\infty$ with $|G|_\infty \in \{ 1, p^\infty \}$, according to $G$ being either a finite group or an infinite $p$-adic group; in the latter case, $|G|_f$ is prime-to-$p$. Let $R$ be a commutative unitary ring. As long as all the primes in $|G|_\infty$ are invertible in $R$, there exists a Haar measure on $G$ with values in $R$, that is a non-zero left $G$-equivariant morphism $C_c^\infty(G,R) \to 1_G^R$ where $C_c^\infty(G,R)$ is the space of locally constant and compactly supported functions in $G$ with values in $R$, and $(1_G^R,R)$ is the trivial representation. A normalised Haar measure on $G$ is a Haar measure taking the value $1$ on a compact open subgroup of $G$. In particular such a compact open subgroup must be of invertible pro-order in $R$. Reciprocally, any normalised Haar measure arises as a Haar measure having value $1$ on a compact open subgroup of invertible pro-order in $R$.

\paragraph{The space $W$.} Let $(W,\langle,\rangle)$ be a symplectic vector space of finite dimension over $F$. Its isometry group is composed of the $F$-linear invertible endomorphisms preserving the form $\langle,\rangle$ and is classically denoted $\textup{Sp}(W)$. A lagrangian in $W$ is a maximal totally isotropic subspace. Denote the dimension of $W$ by $n=2m$, then $X$ is a lagrangian if and only if it is a vector subspace which is totally isotropic (\textit{i.e.} $\forall x, x' \in X, \langle x,x' \rangle=0$) of dimension $m$. A lattice $L$ in $W$ is a free $\mathcal{O}_F$-module of rank $n$. The locally profinite topology on the field $F$ induces a locally profinite topology on the finite dimensional vector space $W$. As a result, a lattice in $W$ is a compact open set. Furthermore the subspace topology induced from that of $\textup{End}_F(W)$ on the symplectic group $\textup{Sp}(W)$ is the locally profinite one as well.

\section{Metaplectic representations over $\mathcal{A}$-algebras}

The Heisenberg group $H$ is the set $W \times F$ endowed with the product topology and the composition law:
$$(w,t) \cdot (w',t') = (w+w',t+t'+\frac{1}{2} \langle w, w' \rangle)$$
for $(w,t)$ and $(w',t')$ in $H = W \times F$.

Let $\mathcal{B}$ be an $\mathcal{A}$-algebra with structure morphism $\varphi$. Let $\psi^\mathcal{A} : F \to \mathcal{A}^\times$ be a non-trivial smooth character. As already mentioned in Section \ref{notations_section}, this defines, by composing $\psi$ and $\varphi$, a character $\psi^\mathcal{B} : F \to \mathcal{B}^\times$ which is smooth and non-trivial.

\subsection{A lemma for representations over rings}

Let $G$ be a group and $\mathcal{R}$ a commutative ring. For every prime ideal $\mathcal{P}$ in $\textup{Spec}(\mathcal{R})$, one denotes $k(\mathcal{P})$ the fraction field of $\mathcal{R}(\mathcal{P}) = \mathcal{R} / \mathcal{P}$. Both $k(\mathcal{P})$ and $\mathcal{R}(\mathcal{P})$ are endowed with an obvious structure of $\mathcal{R}$-algebras. For any $\mathcal{R}[G]$-module $V$, the tensor product $V \otimes_\mathcal{R} k(\mathcal{P})$ is a $k(\mathcal{P})[G]$-module in the obvious way. Of course, the latter is smooth if the former is.

\begin{defi} An $\mathcal{R}[G]$-module $V$ is said to be irreducible at $\mathcal{P} \in \textup{Spec}(\mathcal{R})$ if the representation $V \otimes_\mathcal{R} k(\mathcal{P}) \in \textup{Rep}_{k(\mathcal{P})}(G)$ is irreducible. By extension, $V$ is everywhere irreducible if it is irreducible at any point of $\textup{Spec}(\mathcal{R})$. \end{defi}

There exists a simple sufficient condition to be everywhere irreducible:

\begin{lem} \label{everywhere_irreducible_lemma} Let $V$ be an $\mathcal{R}[G]$-module and consider the map $I \mapsto IV$ that maps an ideal $I$ of $\mathcal{R}$ to the sub-$\mathcal{R}[G]$-module $IV$ of $V$. If the previous map defines a bijection between ideals of $\mathcal{R}$ and sub-$\mathcal{R}|G]$-modules of $V$, then $V$ is everywhere irreducible. \end{lem}

\begin{proof} Using the bijection, one has $\mathcal{P} V \subsetneq V$ for any prime (proper) ideal $\mathcal{P}$, so the module $V \otimes_\mathcal{R} \mathcal{R}(\mathcal{P}) = V / \mathcal{P} V$ is non-zero. It is even $\mathcal{R}(\mathcal{P})$-torsion free because, if $a v \in \mathcal{P} V$ for $a \in R$ and $v \in V$, then $a I_v \subset \mathcal{P}$ where $I_v V = \mathcal{R}[G] \cdot v$. In particular $V \otimes_\mathcal{R} \mathcal{R}(\mathcal{P})$ embeds in $V \otimes_\mathcal{R} k(\mathcal{P})$ by a localisation argument, so the latter representation is non-zero.

In order to prove that $V \otimes_\mathcal{R} k(\mathcal{P})$ is irreducible, let $W$ be a non-zero subrepresentation of $V \otimes_\mathcal{R} k(\mathcal{P})$ and define $W' = \{ v \in V \ | \ v \otimes_\mathcal{R} 1 \in W \}$. As a first elementary claim, this $W'$ is a non-zero sub-$\mathcal{R}[G]$-module of $V$. In addition the bijection hypothesis yields the existence of an ideal $I$ of $\mathcal{R}$ such that $W'=I V$. Observe furthermore thanks to the bijection that $I \subset \mathcal{P}$ if and only if $IV \subset \mathcal{P} V$. As a consequence, the image of $IV$ in $V \otimes_\mathcal{R} k(\mathcal{P})$ generates $V \otimes_\mathcal{R} k(\mathcal{P})$ as a $k(\mathcal{P})$-vector space if and only if $I$ is not contained in $\mathcal{P}$. Of course the image of $W'$ in $V \otimes_\mathcal{R} k(\mathcal{P})$ is non-zero because $W$ is not, so $I$ is not contained in $\mathcal{P}$ \textit{i.e.} the image of $W'$ generates $V \otimes_\mathcal{R} k(\mathcal{P})$. Therefore $W = V \otimes_\mathcal{R} k(\mathcal{P})$.  \end{proof}

\subsection{Models $V_A^\mathcal{B}$ associated to self-dual subgroups} \label{models_associated_to_self_dual_subgp_section}

When $A$ is a closed subgroup of $W$, define:
$$A^\perp = \{ w \in W \ | \ \psi^\mathcal{A}(\langle w, A \rangle) = 1 \}.$$
In this definition, whether one uses $\psi^\mathcal{A}$ or $\psi^\mathcal{B}$ matters not. Now, the closed subgroup $A$ of $W$ is said to be self-dual if $A^\perp = A$. Lagrangians and self-dual lattices provide examples of such subgroups, so there always exist self-dual subgroups in $W$.

\begin{lem} \label{extension_of_psi_B_lem} Let $A$ be a self-dual subgroup of $W$. Then there exists a character $\psi^\mathcal{A}_A$ which extends $\psi^\mathcal{A}$ to the subgroup $A_H = A \times F$ of the Heisenberg group $H$. Furthermore, $\psi^\mathcal{B}_A = \varphi \circ \psi^\mathcal{A}_A$ provides the same kind of extension, that is, $\psi^\mathcal{B}_A$ extends $\psi^\mathcal{B}$ to $A_H$. \end{lem}

This lemma can be proved in the exact same elementary way as  \cite[Lem. 2.2 a)]{trias_theta1}. For the sake of shortness, we simply refer to the latter. The heart of the current section is the following proposition, generalising \cite[Lem. 2.2 b)]{trias_theta1} where the $\mathcal{A}$-algebra $\mathcal{B}$ is a field:

\begin{prop} \label{metaplectic_representations_prop} Let $\psi_A^\mathcal{A}$ be as above and set $V_A^\mathcal{B} = \textup{ind}_{A_H}^H(\psi_A^\mathcal{B}) \in \textup{Rep}_\mathcal{B}(H)$.
\begin{enumerate}[label=\textup{\alph*)}]
\item The map $I \mapsto IV_A^\mathcal{B}$ defines a bijection from the set of ideals of $\mathcal{B}$ to the set of sub-$\mathcal{B}[H]$-modules of $V_A^\mathcal{B}$. In particular, $V_A^\mathcal{B}$ is everywhere irreducible;
\item The $\mathcal{B}[H]$-module $V_A^\mathcal{B}$ is admissible and $V_A^\mathcal{B} = \textup{Ind}_{A_H}^H(\psi_A^\mathcal{B})$;
\item $V_A^\mathcal{B}$ satisfies Schur's lemma, that is $\textup{End}_{\mathcal{B}[H]}(V_A^\mathcal{B}) = \mathcal{B}$.
\end{enumerate} \end{prop}

\begin{proof} The core idea of the proof comes from \cite[Lem. 2.2 b) \& Prop 2.4 c)]{trias_theta1}, which was originally generalising \cite[Chap. 2, I.3 \& I.6]{mvw}. As some differences occur when dealing with $\mathcal{A}$-algebras instead of fields, we carefully examine and detail them below.

\noindent \textbf{a)} First remark that, assuming the bijection property holds, the second part of the statement is a mere application of Lemma \ref{everywhere_irreducible_lemma}. Therefore we focus our attention to proving that such a bijection holds.

The $\mathcal{B}[H]$-module $V_A^\mathcal{B}$ is generated as a $\mathcal{B}$-module by a family $(\chi_{w,L})$ we now describe. As $\psi_A^\mathcal{B}$ is smooth, there exists for all $w \in W$ an open compact subgroup $L_w$ of $W$ such that $\psi_A^\mathcal{B}(a) = 1$ for all $a \in A_H \cap (w,0) (L_w,0)(w,0)^{-1}$. Fix such choices of small enough lattices $(L_w)_{w\in W}$. Then if $L$ is a sublattice of $L_w$, there exists a unique function in $V_A^\mathcal{B}$ which is supported on $A_H (w,0) (L,0)$, right invariant under $(L,0)$ and taking the value~$1$ at $(w,0)$. One denotes it $\chi_{w,L}$. The $\mathcal{B}[H]$-module $V_A^\mathcal{B}$ being smooth, any $f$ in this compactly induced module can be written as a finite sum of such $\chi_{w,L}$, that is the family $(\chi_{w,L})_{w \in W, L \subset L_w}$ is generating $V_A^\mathcal{B}$. For this generation property, $I V_A^\mathcal{B}$ is identified with the set of functions in $V_A^\mathcal{B}$ taking values in $I$ as the family $(i \times \chi_{w,L})_{i \in I, w \in W, L \subset L_w}$ generates both of them. This implies the injectivity for the map $I \mapsto I V_A^\mathcal{B}$.

The surjectivy of $I \mapsto I V_A^\mathcal{B}$ amounts to proving that any sub-$\mathcal{B}[H]$-module of $V_A^\mathcal{B}$ is of the form $I V_A^\mathcal{B}$. For any subset $X$ of $V_A^\mathcal{B}$, define $I_X = < f(h) \ | \ h \in H, f \in X >$ the ideal in $\mathcal{B}$ generated by the set of values of functions in $X$. There is an obvious inclusion of $\mathcal{B}[H]$-modules $\mathcal{B}[H] \cdot X \subset I_X V_A^\mathcal{B}$. We claim even more: this inclusion actually is an equality. It is enough to prove it when $X$ is a singleton to deduce the result general case because $\mathcal{B}[H] \cdot X = \sum \mathcal{B}[H] \cdot f$ and $I_X V_A^\mathcal{B} = \sum I_f V_A^\mathcal{B}$ where the sums run over all $f \in X$. So from now on, suppose that $X$ is made of a single function $f$ in $V_A^\mathcal{B}$. We would like to prove that the reverse inclusion holds, that is:
$$I_f V_A^\mathcal{B} \subset \mathcal{B}[H] \cdot f.$$

As $p$ is invertible in $\mathcal{B}$, there exists a Haar measure of $H$ which takes values in $\mathcal{B}$ and is normalised over a compact open subgroup of $H$. Let $\mu$ be such a measure. The claim will then follow from the -- technical-to-state but rather clear -- observation below:

\begin{lem} \label{metaplectic_rep_small_lemma} Let $f$ be a non-zero function of $V_A^\mathcal{B}$. For any $w \in W$, fix a sufficiently small lattice $L_w$ in $W$ such that $(L_w,0)$ leaves $f$ right invariant and $\psi_A^\mathcal{B}(a)=1$ for all $a \in A_H \cap (w,0) (L_w,0) (w,0)^{-1}$. Then for any sublattice $L$ of $L_w$, there exists an element $\phi_{w,L} \in \mathcal{B}[H]$ such that $\phi_{w,L} \cdot f = f((w,0)) \chi_{w,L}$.  \end{lem}

\begin{proof} First of all, the fact that such a choice of lattices $(L_w)_{w \in W}$ exists comes for the smoothness of $V_A^\mathcal{B}$ and $\psi_A^\mathcal{B}$. Let $L$ be sublattice of $L_w$ and define:
$$\phi : a \in A \mapsto \frac{\psi_A^\mathcal{B}((-a,0))}{\textup{vol}(L^\perp \cap A)} \psi^\mathcal{B}(\langle -w , a \rangle) 1_{L^\perp \cap A}(a) \in \mathcal{B}$$
where $1_X$ is the characteristic funtion of $X$, $\mu_A$ is a Haar measure of $A$ normalised over a compact open subgroup and $\textup{vol}(L^\perp \cap A)$ is a power of $p$. Then an explicit computation will show that the function:
$$\phi \cdot f : h \in H \mapsto \int_A \phi(a) f(h(a,0)) d \mu_A(a) \in \mathcal{B}$$
belongs to $\mathcal{B}[H] \cdot f$ and is a scalar multiple of $\chi_{w,L}$.

We give short but prompt explanation of this last computational claim. Given that the function $\phi$ is compactly supported and locally constant, one can write -- up to some volume factor which is a mere power of $p$ -- the function $\phi \cdot f$ as a finite sum:
$$\sum \phi(a_i) f(h(a_i,0)) = \left(\sum \phi(a_i) (a_i,0)\right) \cdot f (h).$$
So $\phi \cdot f$ belongs to $\mathcal{B}[H] \cdot f$.  For all $w' \in W$, the compution mentioned above reads:
\begin{eqnarray*} 
\phi \cdot f((w',0))& = & f((w',0)) \times \frac{1}{\textup{vol}(L^\perp \cap A)}  \int_{L^\perp \cap A} \psi^\mathcal{B}(\langle w'-w,a \rangle ) d \mu_A(a).
\end{eqnarray*}
A classical argument rewrites the last term as $1_{A+w+L}(w')$. Therefore $\phi \cdot f$ has support $A_H(w,0)(L,0)$, is right invariant under $(L,0)$ and takes the value $f((w,0))$ at $(w,0)$. By unicity, one must have $\phi \cdot f = f((w,0)) \chi_{w,L}$. \end{proof}

Applying the previous lemma, we conclude that the reverse inclusion $I_f V_A^{\mathcal{B}} \subset \mathcal{B}[H] \cdot f$ holds. So the map $I \mapsto I V_A^\mathcal{B}$ is injective and surjective, that is being bijective.

\noindent \textbf{b)} Let $L$ be an open compact subgroup of $W$. Let $w \in W$. Consider the set of functions left $\psi_A^\mathcal{B}$-equivariant, supported on the double coset $A_H (w,0) (L,0)$ and right invariant under $(L,0)$. Actually this space of functions is isomorphic to either $\mathcal{B}$ or $0$ as a consequence of the formula for invariants vectors in compactly induced representations \cite[I.5.6]{vig_book}. Denote by $\chi_{w,L}$ the appropriate generator, meaning the function that takes value either $1$ or $0$ at $(w,0)$. Fix representantives in $W$ for the double coset $A_H \backslash H / (L,0) \simeq A \backslash W / L = W/ (A+L)$. Remark that the admissibility of $V_A^\mathcal{B}$ follows from the fact that, given some $L$, there are only finitely many representantives $w$ giving rise to non-zero functions $\chi_{w,L}$. We are now proving this claim about functions $\chi_{w,L}$.

Suppose $\chi_{w,L}$ is non-zero. For all $l \in L \cap A$, one has:
$$1=\chi_{w,L}((w,0))=\chi_{w,L}((w,0)(l,0))= \chi_{w,L}((l,\langle w,l \rangle )(w,0))=\psi^\mathcal{B}(\langle w,l \rangle )\psi_A^\mathcal{B}((l,0)).$$
Thus for all $l \in L \cap A$, the relation $\psi^\mathcal{B}(\langle w,l \rangle) = \psi_A^{\mathcal{B}}((-l,0))$ must hold. It means that any two representantives $w$ and $w'$, giving rise to non zero $\chi_{w,L}$ and $\chi_{w',L}$, must satisfy the relation $\psi^\mathcal{B}(\langle w-w',l \rangle)=1$ for all $l \in L \cap A$ \textit{i.e.} $w-w' \in (L \cap A)^\perp$. However:
$$(L \cap A)^\perp = L^\perp + A^\perp = L^\perp + A.$$
As $L$ is compact open, its orthogonal $L^\perp$ is compact open too because this holds for lattices in $W$. So the image of $L^\perp$ in the quotient $W/(A+L)$ is a finite set, which means the set of representatives $w$ giving rise to non-zero $\chi_{w,L}$, when $L$ is fixed, is finite.

To conclude, for any sufficiently small open compact subgroup $L$ of $W$, the condition for smallness being $L \times \textup{Ker}(\psi^\mathcal{B})$ is a subgroup of $H$, one has:
$$(V_A^\mathcal{B})^{L \times \textup{Ker}(\psi^\mathcal{B})} = \bigoplus_{\chi_{w,L} \neq 0} \mathcal{B} \cdot \chi_{w,L}$$
where the right-hand side sum is finite. So the smooth $\mathcal{B}[H]$-module $V_A^\mathcal{B}$ is admissible, and according to \cite[I.5.6 1)]{vig_book}, it is equivalent to saying that $\textup{ind}_{A_H}^H(\psi_A^\mathcal{B})  = \textup{Ind}_{A_H}^H(\psi_A^\mathcal{B})$.

\noindent \textbf{c)} As proved in the previous point, there exists a sufficiently small open compact subgroup $L$ of $W$ such that $K=L\times \textup{Ker}(\psi^\mathcal{B})$ is a subgroup of $H$ and:
$$(V_A^\mathcal{B})^K = \bigoplus_{\chi_{w,L} \neq 0} \mathcal{B} \cdot \chi_{w,L}$$
where the right-hand side sum is finite. In addition, there exists a non-zero $\chi_{w,L}$ for $w \in W$ if the condition ``$\psi_A^{\mathcal{B}}(a) = 1$ for all $a \in A_H \cap (w,0)(L,0)(w,0)^{-1}$'' is satisfied. Therefore, up to strengthening the sufficiently small condition, one may suppose that $(V_A^\mathcal{B})^K \neq 0$. Because every $\mathcal{B} \cdot \chi_{w,L}$ is isomorphic to $\mathcal{B}$, and the sum runs over functions with mutually disjoint supports, the $\mathcal{B}$-module $(V_A^\mathcal{B})^K$ is a free module of finite rank.

Thanks to point a), the $\mathcal{B}[H]$-module $V_A^\mathcal{B}$ is generated by a single element $\chi_{w,L}$. Indeed, the ideal $I_{\chi_{w,L}}=< \chi_{w,L}(h) \ | \ h \in H>$ satisfies $\mathcal{B}[H] \cdot \chi_{w,L} = I_{\chi_{w,L}} V_A^\mathcal{B}$ and contains $1$ since $\chi_{w,L}((w,0))=1$. Thus the restriction to $(V_A^\mathcal{B})^K$ induces an injective morphism of $\mathcal{B}$-algebras:
$$\xi : \textup{End}_{B[H]}(V_A^\mathcal{B}) \hookrightarrow  \textup{End}_{\mathcal{H}_\mathcal{B}(H,K)}((V_A^\mathcal{B})^K),$$
where $(V_A^\mathcal{B})^K$ is a module on the relative Hecke algebra $\mathcal{H}_\mathcal{B}(H,K)$ \cite[I.4.5]{vig_book}.

The module $(V_A^\mathcal{B})^K$ being free over $\mathcal{B}$, write its basis $\mathfrak{B}= (\chi_{w,L})_w$. In this basis, the function $\phi_{w,L}$ defined above in the proof of Lemma \ref{metaplectic_rep_small_lemma} becomes the elementary projector $E_w$ onto $\chi_{w,L}$ \textit{i.e.} for all $w' \in \mathfrak{B}$ one has:
$$\phi_{w,L} \cdot \chi_{w',L} = \chi_{w',L}((w,0)) \times \chi_{w,L} = \left\{ \begin{array}{l}
0 \textup{ if } w' \neq w; \\
\chi_{w,L} \textup{ otherwise}. \end{array} \right.$$

Let now $\Phi \in \textup{End}_{\mathcal{B}[H]}(V_A^\mathcal{B})$. Then the image $\xi(\Phi)$ of $\Phi$ in $\textup{End}_{\mathcal{H}_\mathcal{B}(H,K)}((V_A^\mathcal{B})^K)$ commutes with $E_w$ for all $w \in \mathfrak{B}$ as it commutes with the action of $\phi_{w,L}$. Because of this commutation relation between $\xi(\Phi)$ and $E_w$, there exists a scalar $\lambda_w \in \mathcal{B}$ such that $\xi(\Phi)(\chi_{w,L}) = \lambda_w \times \chi_{w,L}$. As any $\chi_{w,L}$ generates $V_A^\mathcal{B}$ as a $\mathcal{B}[H]$-module, it does generate $(V_A^\mathcal{B})^K$ as a $\mathcal{H}_\mathcal{B}(H,K)$-module. This last fact implies that all the $\lambda_w$ are equal. Therefore there exists $\lambda \in \mathcal{B}$ such that $\xi(\Phi) = \lambda \textup{Id}_{(V_A^\mathcal{B})^K}$. So $\Phi = \lambda \textup{id}_{V_A^\mathcal{B}}$ because $\xi$ is injective. \end{proof}

The following can be easily deduced from Proposition \ref{metaplectic_representations_prop} that has just been proved and the finiteness property of the compact induction:

\begin{cor} \label{metaplectic_representations_cor} Let $\mathcal{B}'$ be a $\mathcal{B}$-algebra given by the ring morphism $\varphi' : \mathcal{B} \to \mathcal{B}'$. Then the morphism $\varphi'$ induces a canonical isomorphism of smooth $\mathcal{B}'[H]$-modules:
$$V_A^\mathcal{B} \otimes_\mathcal{B} \mathcal{B}' \simeq V_A^{\mathcal{B}'}.$$
It is given on simple tensor elements by the map $f \otimes_\mathcal{B} b' \mapsto b' \times (\varphi' \circ f)$. \end{cor} 

This result will allow to reduce any problem over an $\mathcal{A}$-algebra to a problem over $\mathcal{A}$, because applying the corollary leads to the canonical identification:
$$V_A^\mathcal{B} \simeq V_A^\mathcal{A} \otimes_\mathcal{A} \mathcal{B}.$$
Furthermore, one can consider $\mathcal{A}$-algebras that are not integral domains. For instance, if $\mathcal{B} = \prod_i \mathcal{B}_i$ is a finite product of $\mathcal{A}$-algebras $(\mathcal{B}_i)$, then:
$$V_A^\mathcal{B} \simeq \bigoplus_i V_A^{\mathcal{B}_i}.$$

\subsection{Changing models from $V_{A_1}^\mathcal{B}$ to $V_{A_2}^\mathcal{B}$} \label{changing_models_section}

Let $A_1$ and $A_2$ be two self-dual subgroups of $W$. Let $\psi_{A_1}^\mathcal{A}$ be a character that extends $\psi^\mathcal{B}$ to $A_{1,H}$ as in Lemma \ref{extension_of_psi_B_lem}. Similarly, fix an extension $\psi_{A_2}^\mathcal{A}$ of $\psi^\mathcal{B}$ with respect to $A_{2,H}$. Once again, set $\psi_{A_1}^\mathcal{B} = \varphi \circ \psi_{A_1}^\mathcal{A}$ and $\psi_{A_2}^\mathcal{B} = \varphi \circ \psi_{A_2}^\mathcal{A}$, which are both smooth and non-trivial characters. Suppose $\omega \in W$ satisfies the condition:
$$\psi_{A_1}^\mathcal{B} ((a,0)) \psi_{A_2}^\mathcal{B} ((a,0))^{-1} = \psi^\mathcal{B}( \langle a ,\omega \rangle) \textup{ for all } a \in A_1 \cap A_2.$$
Note that such an $\omega$ always exist as the left-hand side defines a character of $A_1 \cap A_2$.

Let $\mu$ be a Haar measure with values in $\mathcal{B}$ of the quotient $A_1 \cap A_2 \backslash A_2$. Define:
$$I_\mu = < \mu(K) \ | \ K \textup{ open compact subgroup}>$$
the ideal of $\mathcal{B}$ generated by the various values taken by $\mu$ on the open compact subgroups of $A_1 \cap A_2 \backslash A_2$. By unicity of the Haar measure, the ideal $I_\mu$ is prinicpal and generated by any $\mu(K)$ as long as the pro-order of $K$ is invertible in $\mathcal{B}$. The measure is said to be invertible if $I_\mu = \mathcal{B}$. Of course, every normalised Haar measure, that is a measure taking the value $1$ on a compact open subgroup, is invertible. For $\mu$ to be invertible, it is necessary and sufficient that there exists a compact open subgroup whose measure is a unit in $\mathcal{B}$ \textit{i.e.} $\mu$ is a unit multiple of a normalised Haar measure.

\begin{prop} \label{intertwining_maps_models_A1_A2_prop} The map $I_{A_1,A_2,\mu,\omega}$ associating to $f \in V_{A_1}^\mathcal{B}$ the function:
$$I_{A_1,A_2,\mu,\omega}f : h \longmapsto \int_{A_{1,H} \cap A_{2,H} \backslash A_{2,H}} \psi_{A_2}^\mathcal{B}(a)^{-1} f((\omega,0) a h) \ d\mu(a)$$
is a morphism of smooth $\mathcal{B}[H]$-modules from $V_{A_1}^\mathcal{B}$ to $V_{A_2}^\mathcal{B}$. Its image is $I_{\mu} V_{A_2}^\mathcal{B}$ and, as a result, $I_{A_1,A_2,\mu,\omega}$ is an isomorphism if and only if $\mu$ is an invertible measure. In addition, any invertible measure $\mu$ induces an isomorphism of $\mathcal{B}$-modules:
$$\textup{Hom}_{\mathcal{B}[H]}(V_{A_1}^\mathcal{B},V_{A_2}^\mathcal{B}) = \{ \lambda I_{A_1,A_2,\mu,\omega} \ | \ \lambda \in \mathcal{B} \} \simeq \mathcal{B}.$$ \end{prop}

\begin{proof} On the one hand, the function $I_{A_1,A_2,\mu,\omega} f$ is well defined. Indeed for any $h \in H$, the function $a \in A_{2,H} \mapsto \psi_{A_2}^\mathcal{B}(a)^{-1} f((\omega,0)ah) \in \mathcal{B}$ is $(A_{1_H} \cap A_{2,H})$-left invariant and locally constant, so one can consider it is a function on $A_{1,H} \cap A_{2,H} \backslash A_{2,H} = A_1 \cap A_2 \backslash A_2$. The function $a \in A_{2,H} \mapsto f((\omega,0) a h) \in \mathcal{B}$ is compactly supported modulo $A_{1,H} \cap A_{2,H}$ because, as in \cite[Sec. 2.3]{trias_theta1}, the sum $A_1 + A_2$ is a closed subgroup of $H$. Finally, a change of variables implies that $I_{A_1,A_2,\mu,\omega} f$ is left $\psi_{A_2}^\mathcal{B}$-equivariant. The map $I_{A_1,A_2,\mu,\omega}$ is clearly $\mathcal{B}$-linear and $H$-equivariant so that $I_{A_1,A_2,\mu,\omega} \in \textup{Hom}_{\mathcal{B}[H]}(V_{A_1}^\mathcal{B}, V_{A_2}^\mathcal{B})$.

As a result of point a) from Proposition \ref{metaplectic_representations_prop}, the image of $I_{A_1,A_2,\mu,\omega}$ must be of the form $I V_{A_2}^\mathcal{B}$ for some ideal $I$ in $\mathcal{B}$. Actually, we proved a sharper results in the proof of point a) showing that:
$$I=\{ I_{A_1,A_2,\mu,\omega} f (h) \ | \ f \in V_{A_1}^\mathcal{B}, h \in H \}.$$
If $\mu$ is chosen to be invertible, then for any other measure $\mu'$, there exists $\lambda \in \mathcal{B}$ generating $I_{\mu'}$ and such that the image of $I_{A_1,A_2,\mu',\omega}$ is $I_{\mu'} I V_{A_2}^\mathcal{B} = \lambda I V_{A_2}^\mathcal{B}$. It reduces to consider the morphism $I_{A_1,A_2,\mu,\omega}$ when $\mu$ is invertible. In this case, we show below that the morphism is surjective and as injective.

Suppose $\mu$ is invertible. As in the proof of Proposition \ref{metaplectic_representations_prop}, choose a sufficiently small open compact subgroup $L$ of $W$ such that there exists a non-zero function $\chi_{\omega,L}$ supported on $A_{1,H} (\omega,0) (L,0)$, right invariant under $(L,0)$ and taking the value $1$ at $(\omega,0)$. One may as well suppose that $\psi_{A_2}((l,0)) = 1$ for all $l \in L$, by choosing an even smaller $L$ if needed. Then the formula for $\chi_{\omega,L}$ at $h=(0,0)$ reads:
\begin{eqnarray*} I_{A_1,A_2,\mu,\omega} \chi_{\omega,L}((0,0)) &=& \int_{L \cap A_1 \cap A_2 \backslash L \cap A_2} \psi_{A_2}((l,0))^{-1} \chi_{\omega,L}((\omega,0) (l,0) ) \ d\mu(l) \\
&=& \int_{L\cap A_1 \cap A_2 \backslash L \cap A_2} \chi_{\omega,L}((\omega,0)) \ d\mu(l) \\
& & \\
&=& \textup{vol}(L\cap A_1 \cap A_2 \backslash L \cap A_2). \end{eqnarray*}
The group $L \cap A_1 \cap A_2 \backslash L \cap A_2$ has pro-order a power of $p$, so its volume for the invertible measure $\mu$ is a unit \textit{i.e.} $I_{A_1,A_2,\mu,\omega} \chi_{\omega,L}((0,0)) \in \mathcal{B}^\times$.

Therefore the previous unit $I_{A_1,A_2,\mu,\omega} \chi_{\omega,L}((0,0))$ belongs to $I$ \textit{i.e.} $I=\mathcal{B}=I_\mu$.  It follows that the morphism $I_{A_1,A_2,\mu,\omega}$ is surjective. It is injective as well. Indeed, its kernel is of the form $I' V_{A_1}^\mathcal{B}$ for some ideal $I'$ of $\mathcal{B}$, and for all $i' \in I'$, the function $i' \chi_{\omega,L}$ belongs to the kernel. However the function $I_{A_1,A_2,\mu,\omega} (i' \chi_{\omega,L}) = i' I_{A_1,A_2,\mu,\omega} \chi_{\omega,L}$ takes the value $i'$ at $(0,0)$ and is the zero function. So $i'=0$ and $I'$ is the zero ideal of $\mathcal{B}$. \end{proof}

Consider the scalar extension functor:
$$V \in \textup{Rep}_\mathcal{A}(H) \mapsto V \otimes_\mathcal{A} \mathcal{B} \in \textup{Rep}_\mathcal{B}(H)$$
and denote $\phi_\mathcal{B} : \textup{Hom}_{\mathcal{A}[H]}(V_{A_1}^\mathcal{A},V_{A_2}^\mathcal{A}) \to \textup{Hom}_{\mathcal{B}[H]}(V_{A_1}^\mathcal{B},V_{A_2}^\mathcal{B})$ the map that is induced by functoriality.

In particular for all  $f \in \textup{Hom}_{\mathcal{A}[H]}(V_{A_1}^\mathcal{A},V_{A_2}^\mathcal{A})$, the following diagram, where the vertical arrows are given by the canonical $V_A^\mathcal{A} \to V_A^\mathcal{A} \otimes_\mathcal{A} \mathcal{B}$ of Corollary \ref{metaplectic_representations_cor}, is commutative:
$$\xymatrix{
		V_{A_1}^\mathcal{A} \ar[r]^{f} \ar[d] & V_{A_2}^\mathcal{A}  \ar[d] \\
		V_{A_1}^\mathcal{B} \ar[r]^{\phi_\mathcal{B}(f)} & V_{A_2}^\mathcal{B}
		}.$$
For $\omega \in W$, observe now that the two following conditions are equivalent:
\begin{itemize}[label=$\bullet$]
\item $\psi_{A_1}^\mathcal{A} ((a,0)) \psi_{A_2}^\mathcal{A} ((a,0))^{-1} = \psi^\mathcal{A}( \langle a ,\omega \rangle)$ for all $a \in A_1 \cap A_2$; 
\item $\psi_{A_1}^\mathcal{B} ((a,0)) \psi_{A_2}^\mathcal{B} ((a,0))^{-1} = \psi^\mathcal{B}( \langle a ,\omega \rangle)$ for all $a \in A_1 \cap A_2$.
\end{itemize}
Fix $\omega \in W$ satisfying one of the previous two. The corollary below is quite immediate:

\begin{cor} \label{intertwining_maps_models_A1_A2_cor} Let $\mu^\mathcal{A}$ be an invertible Haar measure of $A_1 \cap A_2 \backslash A_2$ with values in $\mathcal{A}$. Set $\mu^\mathcal{B} = \varphi \circ \mu^\mathcal{A}$. This latter measure is an invertible $\mathcal{B}$-valued measure. Then for all $M \in \textup{Hom}_{\mathcal{B}[H]}(V_{A_1}^\mathcal{B},V_{A_2}^\mathcal{B})$, there exists $\lambda \in \mathcal{B}$ such that:
$$M = \lambda \times I_{A_1,A_2,\mu^\mathcal{B},\omega} = \lambda \times \phi_\mathcal{B}(I_{A_1,A_2,\mu^\mathcal{A},\omega}).$$ \end{cor}

\section{Weil representations over $\mathcal{A}$-algebras} \label{weil_over_A_algebra_section}

Let $\mathcal{B}$ be an $\mathcal{A}$-algebra. Let $A$ a self-dual subgroup of $W$ and $V_A^\mathcal{B}=\textup{ind}_{A_H}^H(\psi_A^\mathcal{B})$ the smooth $\mathcal{B}[H]$-module built in Section \ref{models_associated_to_self_dual_subgp_section}, where $\psi_A^\mathcal{B}$ is an extension of $\psi^\mathcal{B}$ in the way of Lemma \ref{extension_of_psi_B_lem}. The symplectic group $\textup{Sp}(W)$ is naturally acting on $H$ through the first coordinate, that is:
$$g \cdot (w,t) = (g w ,t)$$
for $g \in \textup{Sp}(W)$ and $(w,t) \in H$. Of course, self-dual subgroups are preserved under this action, that is $gA$ is self-dual for all $g \in \textup{Sp}(W)$.

In this section $g$ always denotes an element of $\textup{Sp}(W)$. For $f \in V_A^\mathcal{B}$, the function:
$$I_gf : h \in H \mapsto f(g^{-1} \cdot h) \in \mathcal{B}$$
belongs to $V_{gA}^\mathcal{B} = \textup{ind}_{(gA)_H}^H(\psi_{gA}^\mathcal{B})$ where $\psi_{gA}^\mathcal{B}(g \cdot a ) = \psi_A^\mathcal{B}(a)$ for all $a \in A_H$. It is important to stress that $V_{gA}^\mathcal{B}$ depends on $g$, because even if $gA = A$, one may have that $\psi_{gA}^\mathcal{B} \neq \psi_{A}^\mathcal{B}$ as characters of $A_H$. Another caution is related to the map:
$$I_g : f \in V_A^\mathcal{B} \mapsto I_g f \in V_{gA}^\mathcal{B}$$
that is not a morphism of $\mathcal{B}[H]$-modules. Indeed, for $h_0 \in H$, one has:
$$I_g(h_0 \cdot f ) = (g \cdot h_0)\cdot I_g f$$
whereas $h_0 \cdot (I_g f) = I_g((g^{-1} \cdot h_0) \cdot f)$.

Recall from Section \ref{changing_models_section} that there exists $\omega_g \in W$ such that the condition:
$$\psi_{gA}^\mathcal{B} ((a,0)) \psi_{A}^\mathcal{B} ((a,0))^{-1} = \psi^\mathcal{B}( \langle a, \omega_g \rangle)$$
holds for all $a \in g A \cap A$. Then for any Haar measure $\mu$ of $gA \cap A \backslash A$, one can compose the following morphisms of $\mathcal{B}$-modules:
$$V_A^\mathcal{B} \overset{I_g}{\longrightarrow} V_{gA}^\mathcal{B} \overset{I_{gA,A,\mu,\omega_g}}{\longrightarrow} V_{A}^\mathcal{B}.$$
Therefore $I_{gA,A,\mu,\omega_g} \circ I_g \in \textup{End}_\mathcal{B} (V_A^\mathcal{B})$ is uniquely defined up to a scalar of $\mathcal{B}$, because the morphism $I_{gA,A,\mu,\omega_g}$ is, thanks to Proposition \ref{intertwining_maps_models_A1_A2_prop}.

Consider now the smooth $\mathcal{B}[H]$-module $(\rho_d,\textup{Ind}_F^H(\psi^\mathcal{B}))$ where $F$ is identified with the centre of $H$. All the $\mathcal{B}[H]$-modules $V_A^\mathcal{B}$ naturally embed in the latter because the restriction of $\psi_A^\mathcal{B}$ to $F$ is $\psi^\mathcal{B}$. Under this canonical identification for $V_A^\mathcal{B}$, one has:
$$I_{gA,A,\mu,\omega_g} \circ I_g \circ \rho_d(h) = \rho_d(g \cdot h) \circ I_{gA,A,\mu,\omega_g} \circ I_g.$$
In other words $I_{gA,A,\mu,\omega_g} \circ I_g \in \textup{Hom}_{\mathcal{B}[H]}((\rho_d,V_A^\mathcal{B}),(\rho_d^g,V_A^\mathcal{B}))$ where $\rho_d^g : h \mapsto \rho_d(g \cdot h)$.

Again in Section \ref{changing_models_section}, invertible Haar measures are defined as unit multiples of normalised Haar measures. These exactly are the measures that can take unit values on compact open subgroups. As the linear map $I_g$ is invertible, one easily deduces from Proposition \ref{intertwining_maps_models_A1_A2_prop} that the previous endomorphism is invertible:
\begin{lem} If $\mu$ is invertible, then $I_{gA,A,\mu,\omega_g} \circ I_g \in \textup{GL}_\mathcal{B}(V_A^\mathcal{B})$. \end{lem}

As a result of the lemma, the image of the set $\{ I_{gA,A,\mu,\omega_g} \circ I_g \ | \ \mu \textup{ invertible} \}$ through the quotient map:
$$\textup{\textsc{red}} : \textup{GL}_\mathcal{B} (V_A^\mathcal{B}) \to \textup{GL}_\mathcal{B} (V_A^\mathcal{B})/\mathcal{B}^\times = \textup{PGL}_\mathcal{B}(V_A^\mathcal{B})$$
is well defined. As already mentioned the map $I_{gA,A,\mu,\omega_g} \circ I_g$ is unique up to a scalar, hence this image consists in a singleton; denote by $M_g$ the single element it contains. Remark that $M_g$ does not depend on the choice of $\omega_g$ because $\textup{Hom}_{\mathcal{B}[H]}(V_{gA}^\mathcal{B},V_A^\mathcal{B}) \simeq \mathcal{B}$ once again by Propositon \ref{intertwining_maps_models_A1_A2_prop}, and the set of invertible elements are those in $\mathcal{B}^\times$, which does correspond to the choice of an invertible Haar measure.

The proposition below allows to build Weil representations with coefficients in $\mathcal{B}$.

\begin{prop} \label{weil_rep_def_prop} The map $\sigma_\mathcal{B} : g \in \textup{Sp}(W) \mapsto M_g \in \textup{PGL}_\mathcal{B}(V_A^\mathcal{B})$ is a group morphism and defines a projective representation $V_A^\mathcal{B}$ of $\textup{Sp}(W)$. Using the fibre product construction, it lifts to a representation $\omega_{\psi^\mathcal{B},V_A^\mathcal{B}}$ of a central extension of $\textup{Sp}(W)$ by $\mathcal{B}^\times$ in the following way:
$$\xymatrix{
		\widetilde{\textup{Sp}}_{\psi^\mathcal{B},V_A^\mathcal{B}}^\mathcal{B}(W) \ar[r]^{\omega_{\psi^\mathcal{B},V_A^\mathcal{B}}} \ar[d]_{p_\mathcal{B}} & \textup{GL}_\mathcal{B}(V_A^\mathcal{B})  \ar[d]^{\textup{\textsc{red}}} \\
		\textup{Sp}(W) \ar[r]^{\sigma_\mathcal{B}} & \textup{PGL}_\mathcal{B}(V_A^\mathcal{B})}$$
where $\widetilde{\textup{Sp}}_{\psi^\mathcal{B},V_A^\mathcal{B}}^\mathcal{B}(W)=\textup{Sp}(W) \times_{\textup{PGL}_\mathcal{B}(V_A^\mathcal{B})} \textup{GL}_\mathcal{B}(V_A^\mathcal{B})$ is the fibre product defined by the group morphisms $\sigma_\mathcal{B}$ and $\textup{\textsc{red}}$, together with the projection maps denoted $p_\mathcal{B}$ and $ \omega_{\psi^\mathcal{B},V_A^\mathcal{B}}$. \end{prop}

\begin{proof} The only point that needs explanation is the claim about $\sigma_\mathcal{B}$ being a group morphism. Let $g$ and $g'$ be two elements in $\textup{Sp}(W)$. By definition, there exists an invertible measure $\mu_g$ on $gA \cap A \backslash A$ and an element $\omega_g \in W$ such that:
$$\textup{\textsc{red}}(I_{gA,A,\mu_{g},\omega_g} \circ I_g) = M_g.$$
Respectively, one can write the same type of relation for $M_{g'}$ with some $\mu_{g'}$ and $\omega_{g'}$.

An explicit computation of the composed map $I_g \circ I_{g'A,A,\mu_{g'},\omega_{g'}}$ gives the existence of an invertible measure $\mu$ on $g g' A \cap gA \backslash gA$ and an element $\omega \in W$ such that the commutation relation:
$$I_g \circ I_{g'A,A,\mu_{g'},\omega_{g'}} = I_{gg'A,gA,\mu,\omega} \circ I_g$$
holds. In addition, the morphism:
$$I_{gA,A,\mu_g,\omega} \circ I_{gg'A,gA,\mu,\omega} \in \textup{Hom}_{\mathcal{B}[H]}(V_{g g'A}^\mathcal{B}, \mathcal{V}_A^\mathcal{B})$$
is invertible because each one of the two is. Therefore Proposition \ref{intertwining_maps_models_A1_A2_prop} asserts the existence of an invertible measure $\mu_{g g'}$ on $A \cap g g' A \backslash g g' A$ and an element $\omega_{g g'} \in W$ such that:
$$ I_{gA,A,\mu_g,\omega} \circ I_{gg'A,gA,\mu,\omega} = I_{gA,A,\mu_{g g'},\omega_{g g'}}.$$
The claim hence follows by using the previous two relations and applying $\textup{\textsc{red}}$ to:
$$(I_{gA,A,\mu_{g},\omega_g} \circ I_g) \circ (I_{g'A,A,\mu_{g'},\omega_{g'}} \circ I_{g'}).$$ \end{proof}

\begin{rem} \label{fibre_product_topology_rem} Actually this fibre product makes sense in the category of topological groups in the following setting. Let $\mathcal{B}$ and $V_A^\mathcal{B}$ be endowed with the discrete topology. Then the compact-open topology on $\textup{GL}_\mathcal{B}(V_A^\mathcal{B})$ is generated by the prebasis of open sets $S_{s,s'} = \{ g \in \textup{GL}_\mathcal{B}(V_A^\mathcal{B}) \ | \ gs=s' \}$ for $s$ and $s'$ in $V_A^\mathcal{B}$. Similarly to \cite[Prop. 3.5]{trias_theta1}, one can prove $\textup{\textsc{red}}$ and $\sigma_\mathcal{B}$ are morphisms of topological groups. As a result of the continuity, the fibre product is a locally profinite group for the product topology and the representation $\omega_{\psi^\mathcal{B},V_A^\mathcal{B}}$ is smooth. However, there is an interesting alternative way to prove it and that is developed in the next section. It illustrates the philosophy: any problem related to an $\mathcal{A}$-algebra $\mathcal{B}$ may be brought back to one directly involving $\mathcal{A}$. \end{rem}

Denote by $\phi_\mathcal{B} : \textup{GL}_\mathcal{A}(V_A^\mathcal{A}) \to \textup{GL}_\mathcal{B}(V_A^\mathcal{B})$ the group morphism induced by the extension of scalars and the canonical identification $V_A^\mathcal{A} \otimes_\mathcal{A} \mathcal{B} \simeq V_A^\mathcal{B}$ coming from Corollary \ref{metaplectic_representations_cor}.

\begin{theo} \label{weil_rep_compatib_PhiB_thm} The group morphism $\phi_\mathcal{B}$ induces a morphism of central extensions:
$$\widetilde{\phi_\mathcal{B}} : (g,M) \in \widetilde{\textup{Sp}}_{\psi^\mathcal{A},V_A^\mathcal{A}}^\mathcal{A}(W) \mapsto (g,\phi_\mathcal{B}(M)) \in \widetilde{\textup{Sp}}_{\psi^\mathcal{B},V_A^\mathcal{B}}^\mathcal{B}(W).$$
The image of $\widetilde{\phi_\mathcal{B}}$ is a central extension of $\textup{Sp}(W)$ by $\varphi(\mathcal{A})^\times$ where $\varphi$ is the structure morphism $\varphi : \mathcal{A} \to \mathcal{B}$. Furthemore, the following diagram commutes:
$$\xymatrix{
		\widetilde{\textup{Sp}}_{\psi^\mathcal{A},V_A^\mathcal{A}}^\mathcal{A}(W) \ar[r]^{\omega_{\psi^\mathcal{A},V_A^\mathcal{A}}} \ar[d]^{\widetilde{\phi_\mathcal{B}}} & \textup{GL}_\mathcal{A}(V_A^\mathcal{A})  \ar[d]^{\phi_\mathcal{B}} \\
		\widetilde{\textup{Sp}}_{\psi^\mathcal{B},V_A^\mathcal{B}}^\mathcal{B}(W) \ar[r]^{\omega_{\psi^\mathcal{B},V_A^\mathcal{B}}} & \textup{GL}_\mathcal{B}(V_A^\mathcal{B})
		}.$$ \end{theo}
		
\begin{proof} By definition $(g,M) \in \textup{Sp}(W) \times \textup{GL}_\mathcal{A}(V_A^\mathcal{A})$ belongs to $\widetilde{\textup{Sp}}_{\psi^\mathcal{A},V_A^\mathcal{A}}^\mathcal{A}(W)$ if there exists an invertible Haar measure $\mu$ on $gA \cap A \backslash A$ with values in $\mathcal{A}$ and an element $\omega$ such that $M = I_{gA,A,\mu,\omega} \circ I_g$. Set $\mu^\mathcal{B} = \varphi \circ \mu$. Using the compatibiliy of Corollary \ref{intertwining_maps_models_A1_A2_cor}, the equality $\phi_\mathcal{B}(I_{gA,A,\mu,\omega})=I_{gA,A,\mu^\mathcal{B},\omega}$ holds and defines an isomorphism in $\textup{Hom}_{\mathcal{B}[H]}(V_{gA}^\mathcal{B},V_A^\mathcal{B})$. Hence:
$$\phi_\mathcal{B}(M) = I_{gA,A,\mu^\mathcal{B},\omega} \circ I_g$$
with $\mu^\mathcal{B}$ invertible, that is $(g,\phi_\mathcal{B}(M)) \in \widetilde{\textup{Sp}}_{\psi^\mathcal{B},V_A^\mathcal{B}}^\mathcal{B}(W)$.

The map $\widetilde{\phi_\mathcal{B}}$ thus defined clearly is a morphism of central extensions. In addition, an element $(g,M)$ belongs to its kernel if and only if $g = \textup{Id}_W$ and $\phi_\mathcal{B}(M)=\textup{Id}_{V_A^\mathcal{B}}$. However:
$$\{M \in \textup{GL}_\mathcal{A}(V_A^\mathcal{A}) \ | \ (\textup{Id}_W,M) \in \widetilde{\textup{Sp}}_{\psi^\mathcal{A},V_A^\mathcal{A}}^\mathcal{A}(W) \} = \{ \lambda \textup{Id}_{V_A^\mathcal{A}} \ | \ \lambda \in \mathcal{A}^\times\}.$$
Indeed $M$  must be of the form $I_{gA,A,\mu,\omega} \circ I_g = I_{A,A,\mu,0} = \mu(\{0\}) \times \textup{Id}_{V_A^\mathcal{A}}$ where $\mu$ is an invertible measure of the singleton $\{0\}$, so there exists $\lambda \in \mathcal{B}^\times$ such that $M = \lambda \textup{Id}_{V_A^\mathcal{A}}$. Since $\phi_\mathcal{B}(\lambda \textup{Id}_{V_A^\mathcal{A}}) = \varphi(\lambda) \textup{Id}_{V_A^\mathcal{B}}$, the group $\{(\textup{Id}_W,\lambda \textup{Id}_{V_A^\mathcal{A}}) \ | \  \lambda \in \textup{Ker}(\varphi) \} \simeq \textup{Ker}(\varphi)$ is the kernel sought. The assertion on the image follows from the form of this kernel. \end{proof}

Because of the previous compatibility, many problems over $\mathcal{B}$ reduce to those over the minimal ring $\mathcal{A}$. The corollary to the proposition below illustrates this philosophy:

\begin{prop} Let $A$ and $A'$ be two self-dual subgroups of $W$. Let $\Phi_{A,A'}$ be an isomorphism in $\textup{Hom}_{\mathcal{A}[H]}(V_A^\mathcal{A},V_{A'}^\mathcal{A})$. Then $\Phi_{A,A'}$ induces an isomorphism of central extensions:
$$(g,M) \in \widetilde{\textup{Sp}}_{\psi^\mathcal{A},V_A^\mathcal{A}}^\mathcal{A}(W) \mapsto (g,\Phi_{A,A'} M \Phi_{A,A'}^{-1}) \in \widetilde{\textup{Sp}}_{\psi^\mathcal{A},V_{A'}^\mathcal{A}}^\mathcal{A}(W)$$
compatible with the projections defining the fibre products. In particular, the equivalence class of the representation $\omega_{\psi^\mathcal{A},V_A^\mathcal{A}}$ does not on depend $A$ in the sense that:
$$\Phi_{A,A'} \circ \omega_{\psi^\mathcal{A},V_A^\mathcal{A}}((g,M)) \circ \Phi_{A,A'}^{-1}=\omega_{\psi^\mathcal{A},V_{A'}^\mathcal{A}}((g,\Phi_{A,A'} M \Phi_{A,A'}^{-1}))$$
for all $(g,M) \in \widetilde{\textup{Sp}}_{\psi^\mathcal{A},V_A^\mathcal{A}}^\mathcal{A}(W)$. \end{prop}

\begin{proof} The existence of an isomorphism in $\textup{Hom}_{\mathcal{A}[H]}(V_A^\mathcal{A},V_{A'}^\mathcal{A})$ is a consequence of Proposition \ref{intertwining_maps_models_A1_A2_prop}. One can consider for example any $I_{A,A',\mu,\omega}$ as long as $\mu$ is invertible. The fact that $\Phi_{A,A'}$ induces an isomorphism of central extensions is quite clear when writing down the relations because $\Phi_{A,A'}$ is an isomorphism of $\mathcal{A}[H]$-modules. \end{proof}

From Theorem \ref{weil_rep_compatib_PhiB_thm} and the proposition above, one can deduce the exact same result for coefficients in any $\mathcal{A}$-algebra $\mathcal{B}$. Indeed, applying $\phi_\mathcal{B}$ to the last relation yields:

\begin{cor} \label{canonical_identif_central_ext_B_cor} The equivalence class of $\omega_{\psi^\mathcal{B},V_A^\mathcal{B}}$ does not depend on $A$, in the sense that for any other self-dual subgroup $A'$ of $W$, there exists an isomorphism $\Phi_{A,A'}'$ in $\textup{Hom}_{\mathcal{B}[H]}(V_A^\mathcal{B},V_{A'}^\mathcal{B})$  -- one can take $\phi_\mathcal{B} (\Phi_{A,A'})$ for example -- such that:
$$\Phi_{A,A'}' \circ \omega_{\psi^\mathcal{B},V_A^\mathcal{B}}((g,M)) \circ (\Phi_{A,A'}')^{-1}=\omega_{\psi^\mathcal{B},V_A^\mathcal{B}}((g, \Phi_{A,A'}' M (\Phi_{A,A'}')^{-1}))$$
for all $(g,M) \in \widetilde{\textup{Sp}}_{\psi^\mathcal{B},V_{A'}^\mathcal{B}}^\mathcal{B}(W)$. \end{cor}

\section{The metaplectic group over $\mathcal{A}$}

The notations are those of Section \ref{weil_over_A_algebra_section}. To quickly recall the context: let $\mathcal{B}$ be an $\mathcal{A}$-algebra, let $A$ be a self-dual subgroup of $W$ and $V_A^\mathcal{B}=\textup{ind}_{A_H}^H(\psi_A^\mathcal{B})$ be the smooth $\mathcal{B}[H]$-module built in Section \ref{models_associated_to_self_dual_subgp_section}, where $\psi_A^\mathcal{B}$ is an extension of $\psi^\mathcal{B}$ in the way of Lemma \ref{extension_of_psi_B_lem}.

In Section \ref{weil_over_A_algebra_section}, we constructed a projective representation $\sigma_\mathcal{B} : \textup{Sp}(W) \to  \textup{PGL}_\mathcal{B}(V_A^\mathcal{B})$ of the symplectic group and, in Proposition \ref{weil_rep_def_prop}, we lifted it to a representation $(\omega_{\psi^\mathcal{B},V_A^\mathcal{B}},V_A^\mathcal{B})$ of a central extension of $\textup{Sp}(W)$ by $\mathcal{B}^\times$, namely:
$$\omega_{\psi^\mathcal{B},V_A^\mathcal{B}} : \widetilde{\textup{Sp}}_{\psi^\mathcal{B},V_A^\mathcal{B}}^\mathcal{B}(W) \to \textup{GL}_\mathcal{B}(V_A^\mathcal{B}).$$
Recall that the group on the left-hand side is the fibre product in the category of groups of the group morphisms $\sigma_\mathcal{B} : \textup{Sp}(W) \to \textup{PGL}_\mathcal{B}(V_A^\mathcal{B})$ and $\textup{\textsc{red}} : \textup{GL}_\mathcal{B}(V_A^\mathcal{B}) \to \textup{PGL}_\mathcal{B}(V_A^\mathcal{B})$, together with the projection maps $p_\mathcal{B}$ and $\omega_{\psi^\mathcal{B},V_A^\mathcal{B}}$. As a result of this construction, it is a subgroup of $\textup{Sp}(W) \times  \textup{GL}_\mathcal{B}(V_A^\mathcal{B})$. In particular, these constructions make sense over $\mathcal{A}$ itself, and Theorem \ref{weil_rep_compatib_PhiB_thm} completes the picture relating the constructions over $\mathcal{A}$ and over any $\mathcal{A}$-algebra $\mathcal{B}$, yielding a morphism of central extensions:
$$\widetilde{\phi_\mathcal{B}} :  \widetilde{\textup{Sp}}_{\psi^\mathcal{A},V_A^\mathcal{A}}^\mathcal{A}(W) \to \widetilde{\textup{Sp}}_{\psi^\mathcal{B},V_A^\mathcal{B}}^\mathcal{B}(W)$$
compatible with the respective projection maps.

\subsection{A bit of tolopolgy} \label{a_bit_of_top_section}

This section will shed some light on Remark \ref{fibre_product_topology_rem} by bringing topology into the construction of Proposition \ref{weil_rep_def_prop}. Endow $\mathcal{B}$ and $V_A^\mathcal{B}$ with the discrete topology. Then the open-compact topology on $\textup{GL}_\mathcal{B}(V_A^\mathcal{B})$ is generated by the prebasis $S_{s,s'} = \{ M \in \textup{GL}_\mathcal{B}(V_A^\mathcal{B}) \ | \ Ms =s' \}$ for $s$ and $s'$ running through $V_A^\mathcal{B}$.

The group $\textup{PGL}_\mathcal{B}(V_A^\mathcal{B})$ inherits the quotient topology, which is the finest making the quotient map $\textup{\textsc{red}} : \textup{GL}_\mathcal{B}(V_A^\mathcal{B}) \to \textup{PGL}_\mathcal{B}(V_A^\mathcal{B})$ continuous. Recall from Proposition \ref{weil_rep_compatib_PhiB_thm} that the projective representation $\sigma_\mathcal{B} : \textup{Sp}(W) \to  \textup{PGL}_\mathcal{B}(V_A^\mathcal{B})$ was defined in terms of the action of $\textup{Sp}(W)$ on $H$.

\paragraph{The complex case.} The best-known feature comes when $\mathcal{B}$ is the field of complex numbers. Endowing $\mathbb{C}$ with a structure of $\mathcal{A}$-algebra amounts to fixing an embedding $\varphi : \mathcal{A} \to \mathbb{C}$. Observe that all such embeddings have the same image in $\mathbb{C}$, because $\mathcal{K}  / \mathbb{Q}$ is a Galois extension. In particular, the image of the map $\mathcal{A}^\times \to \mathbb{C}^\times$ induced by $\varphi$ does not depend on the choice of $\varphi$.

So when $\mathcal{B}=\mathbb{C}$ and $\varphi$ is fixed, the representation $V_A^\mathbb{C} \in \textup{Rep}_\mathbb{C}(H)$ is irreducible as an application of Stone-von Neumann's theorem \cite[Chap. 2, Th. I.2]{mvw} and : 
\begin{itemize}[label=$\bullet$]
\item $\omega_{\psi^\mathbb{C},V_A^\mathbb{C}}$ is the Weil representation of the metaplectic group $\widetilde{\textup{Sp}}_{\psi^\mathbb{C},V_A^\mathbb{C}}^\mathbb{C}(W)$.
\end{itemize}
The complex theory asserts that the Weil representation is smooth and the metaplectic group is a natural topological subgroup of $\textup{Sp}(W) \times \textup{GL}_\mathbb{C}(V_A^\mathbb{C})$. To be more precise, the metaplectic group is a locally profinite group. Regarding the smoothness condition, this is equivalent to saying that the map $\omega_{\psi^\mathbb{C},V_A^\mathbb{C}}$ is continuous.

These topological properties are consequences of the continuity of the map $\sigma_\mathbb{C}$, which really is the cornerstone of the theory; and the metaplectic group inherits a natural topology as the fibre product in the category of topological groups of the continuous group morphisms $\textup{\textsc{red}}$ and $\sigma_\mathbb{C}$.

\paragraph{Over $\mathcal{A}$.} By analogy, one calls $\widetilde{\textup{Sp}}_{\psi^\mathcal{A},V_A^\mathcal{A}}^\mathcal{A}(W)$ the metaplectic group over $\mathcal{A}$. Referring to Theorem \ref{weil_rep_compatib_PhiB_thm}, it is a subgroup of the metaplectic group because the group morphism:
$$\widetilde{\phi_\mathbb{C}} : (g,M) \in \widetilde{\textup{Sp}}_{\psi^\mathcal{A},V_A^\mathcal{A}}^\mathcal{A}(W) \to (g,\phi_\mathbb{C}(M)) \in \widetilde{\textup{Sp}}_{\psi^\mathbb{C},V_A^\mathbb{C}}^\mathbb{C}(W)$$
is injective. 

\begin{lem} \label{map_phiC_homeo_lem} The map $\phi_\mathbb{C} : M \in \textup{GL}_\mathcal{A}(V_A^\mathcal{A}) \to \phi_\mathbb{C} (M) \in \textup{GL}_\mathbb{C}(V_A^\mathbb{C})$, coming from the scalar extension to $\mathbb{C}$, is continuous and defines an homeomorphism onto its image. \end{lem}

\begin{proof} The image of $\phi_\mathbb{C}$ is endowed with the subspace topology from $\textup{GL}_\mathbb{C}(V_A^\mathbb{C})$. The map $\phi_\mathbb{C}$ is continuous and injective, so it defines a bijection to its image, say $G_\mathcal{A}$. Denote $\phi_\mathbb{C}' : G_\mathcal{A} \to \textup{GL}_\mathcal{A}(V_A^\mathcal{A})$ the inverse map. Then for all $s$ and $s'$ in $V_A^\mathcal{A}$, one has:
\begin{eqnarray*} (\phi_\mathbb{C}')^{-1} (S_{s,s'}) &=& \{ \phi_\mathbb{C} (M) \ | \ M \in \textup{GL}_\mathcal{A}(V_A^\mathcal{A}) \textup{ and } \phi_\mathbb{C}(M) (s \otimes_\mathbb{C}1)=s'\otimes_\mathbb{C} 1\} \\
 &=&  G_\mathcal{A} \cap S_{s\otimes_\mathbb{C}1,s'\otimes_\mathbb{C}1} \end{eqnarray*}
that is the trace of an open set. So $(\phi_\mathbb{C}')^{-1} (S_{s,s'})$ is open in $G_\mathcal{A}$ and $\phi_\mathbb{C}'$ is continuous. \end{proof}

Of course, the embedding $\textup{Sp}(W) \times \textup{GL}_\mathcal{A}(V_A^\mathcal{A}) \to \textup{Sp}(W) \times \textup{GL}_\mathbb{C}(V_A^\mathbb{C})$ induced by $\phi_\mathbb{C}$ is an homeomorphism onto its image as well. As a result of the lemma, the subspace topology on $\textup{Sp}(W) \times \textup{GL}_\mathcal{A}(V_A^\mathcal{A})$, inherited from that of $\textup{Sp}(W) \times \textup{GL}_\mathbb{C}(V_A^\mathbb{C})$ using the previous embedding, coincides with the usual product topology. Restricting this morphism to the metaplectic group over $\mathcal{A}$, which is a subgoup of $\textup{Sp}(W) \times \textup{GL}_\mathcal{A}(V_A^\mathcal{A})$, exactly yields $\widetilde{\phi_\mathbb{C}}$. Because of the homeomorphism property, one can identify the metaplectic group over $\mathcal{A}$ and its image under $\widetilde{\phi_\mathbb{C}}$, resulting in:

\begin{cor} \label{met_gp_over_A_over_C_cor} The group $\widetilde{\textup{Sp}}_{\psi^\mathcal{A},V_A^\mathcal{A}}^\mathcal{A}(W)$ is a topological subgroup of $\widetilde{\textup{Sp}}_{\psi^\mathbb{C},V_A^\mathbb{C}}^\mathbb{C}(W)$, the inclusion being canonically given by $\widetilde{\phi_\mathbb{C}}$. In addition $\widetilde{\phi_\mathbb{C}}$ is an open embedding. \end{cor}

\begin{proof} The fact that it is a topological subgroup follows from Lemma \ref{map_phiC_homeo_lem} and the subsequent discussion. The map $\widetilde{\phi_\mathbb{C}}$ is open because its image is open in the metaplectic group. Indeeed the first projection of the fibre product yields an exact sequence:
$$1 \to \mathbb{C}^\times \to \widetilde{\textup{Sp}}_{\psi^\mathbb{C},V_A^\mathbb{C}}^\mathbb{C}(W) \to \textup{Sp}(W) \to 1.$$
Because $\mathcal{K}/\mathbb{Q}$ is a Galois extension, the image $\varphi(\mathcal{A}^\times)$ of $\mathcal{A}^\times$ does not depend on $\varphi$ and always contains $\{ \pm 1 \}$. As a result, the following diagram is commutative:
\begin{center} \begin{tikzcd} 
1  \ar[r] &  \mathbb{C}^\times \ar[r] & 	\widetilde{\textup{Sp}}_{\psi^\mathbb{C},V_A^\mathbb{C}}^\mathbb{C}(W) \ar[r] & \textup{Sp}(W)  \ar[r] & 1 \\
1 \ar[r] &  \mathcal{A}^\times \ar[u,"\varphi"] \ar[r] &  	\widetilde{\textup{Sp}}_{\psi^\mathcal{A},V_A^\mathcal{A}}^\mathcal{A}(W) \ar[u,"\widetilde{\phi_\mathbb{C}}"] \ar[r] & \textup{Sp}(W)  \ar[u,"\textup{Id}_{\textup{Sp}(W)}"] \ar[r] & 1 \end{tikzcd} \end{center}
and the group $\widetilde{\textup{Sp}}_{\psi^\mathcal{A},V_A^\mathcal{A}}^\mathcal{A}(W)$ contains the reduced metaplectic group $\widehat{\textup{Sp}}_{\psi^\mathbb{C},V_A^\mathbb{C}}^\mathbb{C}(W)$, that is the derived group of the metaplectic group.

When $F$ is local non-archimedean, this is the unique subgroup of the metaplectic group fitting into the exact sequence:
$$1 \to \{ \pm 1 \} \to \widehat{\textup{Sp}}_{\psi^\mathbb{C},V_A^\mathbb{C}}^\mathbb{C}(W) \to \textup{Sp}(W) \to 1.$$
Furthermore this reduced metaplectic group is open in the metatplectic group, so the claim follows because the metaplectic group over $\mathcal{A}$ contains it. When $F$ is finite, the topology can just be ignored as these groups are finite and have discrete topolgy. \end{proof}

As above, denote by $\widehat{\textup{Sp}}_{\psi^\mathbb{C},V_A^\mathbb{C}}^\mathbb{C}(W)$ the derived group of $\widetilde{\textup{Sp}}_{\psi^\mathbb{C},V_A^\mathbb{C}}^\mathbb{C}(W)$. When $F$ is finite, this group is the derived group of $\textup{Sp}(W)$. Except in the exceptional case $F=\mathbb{F}_3$ and $\textup{dim} (W)=2$, the symplectic group is perfect \textit{i.e.} equal to its own derived subgroup. When $F$ is local archimedean it is the so-called reduced metaplectic group, which is a non-trivial extension of $\textup{Sp}(W)$ by $\{ \pm 1 \}$. Actually there exists a unique such (open) subgroup in the metaplectic group. Regardless of what $F$ may be, we use brackets to define the derived group:
$$\widehat{\textup{Sp}}_{\psi^\mathcal{A},V_A^\mathcal{A}}^\mathcal{A}(W) = [\widetilde{\textup{Sp}}_{\psi^\mathcal{A},V_A^\mathcal{A}}^\mathcal{A}(W),\widetilde{\textup{Sp}}_{\psi^\mathcal{A},V_A^\mathcal{A}}^\mathcal{A}(W)].$$
Recall that $\widetilde{\phi_\mathbb{C}}$ canonically identifies $\widetilde{\textup{Sp}}_{\psi^\mathcal{A},V_A^\mathcal{A}}^\mathcal{A}(W)$ with its image in $\widetilde{\textup{Sp}}_{\psi^\mathbb{C},V_A^\mathbb{C}}^\mathbb{C}(W)$. It also induces, by restriction, a map between the respective derived groups.

\begin{prop} \label{metaplectic_group_over_A_prop} One has the following properties:
\begin{enumerate}[label=\textup{\alph*)}]
\item the map $\sigma_\mathcal{A}$ is continuous and $\widetilde{\textup{Sp}}_{\psi^\mathcal{A},V_A^\mathcal{A}}^\mathcal{A}(W)$ is the fibre product in the category of topological groups of the continuous morphisms $\textup{\textsc{red}}$ and $\sigma_\mathcal{A}$;
\item the representation $\omega_{\psi^\mathcal{A},V_A^\mathcal{A}} : \widetilde{\textup{Sp}}_{\psi^\mathcal{A},V_A^\mathcal{A}}^\mathcal{A}(W)  \to \textup{GL}_\mathcal{A}(V_A^\mathcal{A})$ is smooth as this group morphism is the second projection of the fibre product;
\item the group $\widetilde{\textup{Sp}}_{\psi^\mathcal{A},V_A^\mathcal{A}}^\mathcal{A}(W)$ is open in $\widetilde{\textup{Sp}}_{\psi^\mathbb{C},V_A^\mathbb{C}}^\mathbb{C}(W)$ and therefore the metaplectic group over $\mathcal{A}$ is locally profinite;
\item the map $\widetilde{\phi_\mathbb{C}}$ restricts to an isomorphism $\widehat{\textup{Sp}}_{\psi^\mathcal{A},V_A^\mathcal{A}}^\mathcal{A}(W) \simeq \widehat{\textup{Sp}}_{\psi^\mathbb{C},V_A^\mathbb{C}}^\mathbb{C}(W)$ and when:
\begin{enumerate} \item[\textup{i)}] $F$ is finite, it is the symplectic group except when $F =\mathbb{F}_3$ and $\textup{dim}(W)=2$;
\item[\textup{ii)}] $F$ is local non-archimedean, it is the reduced metaplectic group. \end{enumerate} 
\end{enumerate} \end{prop}

\begin{proof} \textbf{a)} The map $\sigma_\mathcal{A}$ is continuous, because $\sigma_\mathbb{C}$ itself is, and one has:
$$\sigma_\mathcal{A} = \overline{\phi_\mathbb{C}} \circ \sigma_\mathbb{C}$$
where $\overline{\phi_\mathbb{C}} : \textup{PGL}_\mathcal{A}(V_A^\mathcal{A}) \to \textup{PGL}_\mathbb{C}(V_A^\mathbb{C})$ is the continuous group morphism defined from $\phi_\mathbb{C}$ by passing to the quotient. The fibre product of $\sigma_\mathcal{A}$ and $\textup{\textsc{red}}$ in the category of topological groups defines a topological subgroup of $\textup{Sp}(W) \times \textup{GL}_\mathcal{A}(V_A^\mathcal{A})$. In particular, this fibre product is, as a group, the metaplectic group over $\mathcal{A}$.

\noindent \textbf{b)} The projection maps are continuous by definition of the fibre product.

\noindent \textbf{c)} As a direct consequence of $\widetilde{\phi_\mathbb{C}}$ being an open embedding, the group $\widetilde{\textup{Sp}}_{\psi^\mathcal{A},V_A^\mathcal{A}}^\mathcal{A}(W)$ is an open subgroup of the metaplectic group, which is locally profinite. Hence it is a closed subgroup, so the subspace topology is the locally profinite one.

\noindent \textbf{d)} The isomorphism follows considering the first projection $p_\mathcal{A} : \widehat{\textup{Sp}}_{\psi^\mathcal{A},V_A^\mathcal{A}}^\mathcal{A}(W)  \to \textup{Sp}(W)$. This map is surjective, as well as $p_\mathbb{C}$. In addition one has the equality:
$$p_\mathbb{C} \circ \widetilde{\phi_\mathbb{C}}= p_\mathcal{A} .$$
Passing to derived groups yields:
$$D(p_\mathbb{C}) : \widehat{\textup{Sp}}_{\psi^\mathbb{C},V_A^\mathbb{C}}^\mathbb{C}(W) \to [\textup{Sp}(W),\textup{Sp}(W)].$$
It is an isomorphism in case i) and a surject morphism of kernel $\{ \pm 1 \}$ for ii). But through the identification given by $\widetilde{\phi_\mathbb{C}}$, one has the inclusion:
$$\widehat{\textup{Sp}}_{\psi^\mathcal{A},V_A^\mathcal{A}}^\mathcal{A}(W) \subset \widehat{\textup{Sp}}_{\psi^\mathbb{C},V_A^\mathbb{C}}^\mathbb{C}(W)$$
and $D(p_\mathbb{C}) \circ \widetilde{\phi_\mathbb{C}}$ is surjective. In case i), the previous inclusion is an equality and except in the exceptional case mentioned the symplectic group is perfect. In case ii), this implies the following inequality for the index of the quotient:
$$[\widehat{\textup{Sp}}_{\psi^\mathbb{C},V_A^\mathbb{C}}^\mathbb{C}(W):\widehat{\textup{Sp}}_{\psi^\mathcal{A},V_A^\mathcal{A}}^\mathcal{A}(W)] \leq 2.$$
It must be $2$ as the reduced metaplectic group can not be split over $\textup{Sp}(W)$. \end{proof}

\paragraph{Over $\mathcal{B}$.} Call $\widetilde{\textup{Sp}}_{\psi^\mathcal{B},V_A^\mathcal{B}}^\mathcal{B}(W)$ the metaplectic group over $\mathcal{B}$ and define its derived group:
$$\widehat{\textup{Sp}}_{\psi^\mathcal{B},V_A^\mathcal{B}}^\mathcal{B}(W) =[\widetilde{\textup{Sp}}_{\psi^\mathcal{B},V_A^\mathcal{B}}^\mathcal{B}(W),\widetilde{\textup{Sp}}_{\psi^\mathcal{B},V_A^\mathcal{B}}^\mathcal{B}(W)].$$
As above, the morphism of central extensions of Theorem \ref{weil_rep_compatib_PhiB_thm}: 
$$\widetilde{\phi_\mathcal{B}} : (g,M) \in \widetilde{\textup{Sp}}_{\psi^\mathcal{A},V_A^\mathcal{A}}^\mathcal{A}(W) \to (g,\phi_\mathcal{B}(M)) \in \widetilde{\textup{Sp}}_{\psi^\mathcal{B},V_A^\mathcal{B}}^\mathcal{B}(W)$$
retricts to a morphism at the level of derived groups. As $\phi_\mathcal{B}$ is continuous, it defines a continuous map $\overline{\phi_\mathcal{B}} : \textup{PGL}_\mathcal{A}(V_A^\mathcal{A}) \to \textup{PGL}_\mathcal{B}(V_A^\mathcal{B})$ at the level of quotients. Then one has the equality $\sigma_\mathcal{B} = \overline{\phi_\mathcal{B}} \circ \sigma_\mathcal{A}$ and one deduces from Propostion \ref{metaplectic_group_over_A_prop} that $\sigma_\mathcal{B}$ is continuous.

\begin{prop} \label{metaplectic_group_over_B_prop} One has the following properties:
\begin{enumerate}[label=\textup{\alph*)}]
\item the group $\widetilde{\textup{Sp}}_{\psi^\mathcal{B},V_A^\mathcal{B}}^\mathcal{B}(W)$ is the fibre product in the category of topological group of the continuous morphisms $\sigma_\mathcal{B}$ and $\textup{\textsc{red}}$, its topology being the subspace topology in $\textup{Sp}(W) \times \textup{GL}_\mathcal{B}(V_A^\mathcal{B})$;
\item the representation $\omega_{\psi^\mathcal{B},V_A^\mathcal{B}} : \widetilde{\textup{Sp}}_{\psi^\mathcal{B},V_A^\mathcal{B}}^\mathcal{B}(W)  \to \textup{GL}_\mathcal{B}(V_A^\mathcal{B})$ is smooth as this group morphism is the second projection of the fibre product;
\item the map $\widetilde{\phi_\mathcal{B}} : (g,M) \in \widetilde{\textup{Sp}}_{\psi^\mathcal{A},V_A^\mathcal{A}}^\mathcal{A}(W) \to (g,\phi_\mathcal{B}(M)) \in \widetilde{\textup{Sp}}_{\psi^\mathcal{B},V_A^\mathcal{B}}^\mathcal{B}(W)$ is an open continuous map and therefore the metapletic group over $\mathcal{B}$ is locally profinite;
\item considering derived groups, the map $\widetilde{\phi_\mathcal{B}}$ restricts to:
\begin{enumerate} \item[i)] a surjection $\widehat{\textup{Sp}}_{\psi^\mathcal{A},V_A^\mathcal{A}}^\mathcal{A}(W) \to \widehat{\textup{Sp}}_{\psi^\mathcal{B},V_A^\mathcal{B}}^\mathcal{B}(W)$ of kernel $\{\pm 1 \}$ and image isomorphic to $\textup{Sp}(W)$ if $F$ is local non-archimedean and $\textup{char}(\mathcal{B})=2$;
\item[ii)] an isomorphism $\widehat{\textup{Sp}}_{\psi^\mathcal{A},V_A^\mathcal{A}}^\mathcal{A}(W) \simeq \widehat{\textup{Sp}}_{\psi^\mathcal{B},V_A^\mathcal{B}}^\mathcal{B}(W)$ otherwise. \end{enumerate} \end{enumerate} \end{prop}

\begin{proof} \textbf{a) b)} Obvious from the definition of fibre products and projections.

\noindent \textbf{c)} This needs some explanation however. Once again when $F$ is finite, the topology is dicrete and the statement trivially holds. Suppose now that $F$ is local non-archimedean. As a first observation, remark that the equality $\phi_\mathcal{B} \circ \omega_{\psi^\mathcal{A},V_A^\mathcal{A}} = \omega_{\psi^\mathcal{B},V_A^\mathcal{B}} \circ \widetilde{\phi_\mathcal{B}}$ holds.

Let $v \in V_A^\mathcal{A}$ such that $v \otimes_\mathcal{A} 1 \in V_A^\mathcal{B}$ is non-zero. For the previous equality is holding, the stabiliser of $v \otimes_\mathcal{A} 1$ will be contained in the image of $\widetilde{\phi_\mathcal{B}}$ as a result of the following two facts. First, one has :
$$\omega_{\psi^\mathcal{B},V_A^\mathcal{B}}(g,\lambda M) (v \otimes_\mathcal{A} 1)= \lambda M (v \otimes_\mathcal{A} 1)$$
for all $(g,M) \in \widetilde{\textup{Sp}}_{\psi^\mathcal{B},V_A^\mathcal{B}}^\mathcal{B}(W)$ and $\lambda \in \mathcal{B}^\times$. Not much has been said so far. Second, the surjectivy of $p_\mathcal{A}$ and $p_\mathcal{B}$ onto $\textup{Sp}(W)$ implies that for all $(g,M) \in \widetilde{\textup{Sp}}_{\psi^\mathcal{B},V_A^\mathcal{B}}^\mathcal{B}(W)$, there exists $\lambda \in \mathcal{B}^\times$ such that $(g,\lambda M)$ is in the image of $\widetilde{\phi_\mathcal{B}}$.

Combining the previous two facts, the stabiliser of $v\otimes_\mathcal{A} 1$ must be included in the image of $\widetilde{\phi_\mathcal{B}}$. So the image of $\widetilde{\phi_\mathcal{B}}$ is open because the stabiliser of any element is open as a consequence of $\omega_{\psi^\mathcal{B},V_A^\mathcal{B}}$ being smooth.

The image of $\widetilde{\phi_\mathcal{B}}$ is an open subgroup in the metaplectic group over $\mathcal{B}$. If this subgroup is a locally profinite group, then the metaplectic group will be too. Using Theorem \ref{weil_rep_compatib_PhiB_thm}, one has an exact sequence:
$$1 \to \textup{Ker}(\widetilde{\phi_\mathcal{B}}) \to \widetilde{\textup{Sp}}_{\psi^\mathcal{A},V_A^\mathcal{A}}^\mathcal{A}(W) \overset{\widetilde{\phi_\mathcal{B}}}{\to} \textup{Im}(\widetilde{\phi_\mathcal{B}}) \to 1.$$
where $\textup{Ker}(\widetilde{\phi_\mathcal{B}}) = \{ (\textup{Id}_W,\lambda \textup{Id}_{V_A^\mathcal{A}}) \ | \ \lambda \in \mathcal{A}^\times \textup{ and } \varphi(\lambda) = 1 \} \simeq \textup{Ker}(\mathcal{A}^\times \to \varphi(\mathcal{A})^\times)$ is a discrete subgroup, so a closed subgroup. Thanks to Proposition \ref{metaplectic_group_over_A_prop} the metaplectic group over $\mathcal{A}$ is locally profinite, so its quotient by the previous discrete subgroup is locally profinite and $\widetilde{\phi_\mathcal{B}}$ factors through it, inducing an homeomorphism of topological groups.

\noindent \textbf{d)} First of all, there is an induced map between derived subgroups:
$$D(\widetilde{\phi_\mathcal{B}}) : \widehat{\textup{Sp}}_{\psi^\mathcal{A},V_A^\mathcal{A}}^\mathcal{A}(W) \to \widehat{\textup{Sp}}_{\psi^\mathcal{B},V_A^\mathcal{B}}^\mathcal{B}(W).$$
But $p_\mathcal{B} \circ D(\widetilde{\phi_\mathcal{B}}) = D(p_\mathcal{A})$ is a surjective map $\widehat{\textup{Sp}}_{\psi^\mathcal{A},V_A^\mathcal{A}}^\mathcal{A}(W) \to [\textup{Sp}(W),\textup{Sp}(W)]$, which is an isomorphism in case i) and has kernel $\{ \pm 1 \}$ in case ii) according to Proposition \ref{metaplectic_group_over_A_prop}. Therefore:
$$\widetilde{\textup{Sp}}_{\psi^\mathcal{B},V_A^\mathcal{B}}^\mathcal{B}(W) / \textup{Im}(D(\widetilde{\phi_\mathcal{B}}))$$
is abelian. By minimality of the derived group, we must have $\textup{Im}(D(\widetilde{\phi_\mathcal{B}})) = \widehat{\textup{Sp}}_{\psi^\mathcal{B},V_A^\mathcal{B}}^\mathcal{B}(W)$. Furthermore:
$$\textup{Ker}(D(\widetilde{\phi_\mathcal{B}})) = \{ (\textup{Id}_W,\lambda \textup{Id}_{V_A^\mathcal{A}}) \ | \ \lambda \in \mathcal{A}^\times \textup{ and } \varphi(\lambda) = 1 \} \cap \widehat{\textup{Sp}}_{\psi^\mathcal{A},V_A^\mathcal{A}}^\mathcal{A}(W).$$

When $F$ is finite, the group $\textup{Ker}(D(\widetilde{\phi_\mathcal{B}}))= \{ (\textup{Id}_W,\textup{Id}_{V_A^\mathcal{A}}) \}$ is trivial. When $F$ is local non-archimedean, it is included in $\{ (\textup{Id}_W,\epsilon \textup{Id}_{V_A^\mathcal{A}}) \ | \ \epsilon \in \{ \pm 1 \} \} \simeq \{ \pm 1 \}$. But this kernel is non-trivial if and only if $\varphi(-1)=\varphi(1)=1$ in $\mathcal{B}$, that is $\varphi(2)=0$, and $\textup{char}(\mathcal{B}) = 2$. \end{proof}

\begin{defi} \label{reduced_met_group_hatA_to_tildeB_def} Let $\widehat{\phi_\mathcal{B}} : \widehat{\textup{Sp}}_{\psi^\mathcal{A},V_A^\mathcal{A}}^\mathcal{A}(W) \to \widetilde{\textup{Sp}}_{\psi^\mathcal{B},V_A^\mathcal{B}}^\mathcal{B}(W)$ be the restriction $\widetilde{\phi_\mathcal{B}}|_{\widehat{\textup{Sp}}_{\psi^\mathcal{A},V_A^\mathcal{A}}^\mathcal{A}(W)}$. \end{defi}

This map will be used later on. Proposition \ref{metaplectic_group_over_B_prop} has already given some key properties of this map: just to mention a few, it is an open map and its kernel is explicit.

\subsection{Reduced cocycle for $\mathcal{A}$-algebras} \label{reduced_cocycle_over_B_section}

One deduces from Proposition \ref{metaplectic_group_over_B_prop} that the metaplectic group over $\mathcal{B}$ either:
\begin{itemize}
\item contains the symplectic group as a subgroup, in which case $\textup{char}(\mathcal{B})=2$ or $F$ is finite;
\item does not contain the symplectic group as a subgroup, in which case $F$ is local non-archimedean and $\textup{char}(\mathcal{B}) \neq 2$, and its derived group is canonically isomorphic to the so-called reduced metaplectic group.
\end{itemize}

\subsubsection{Non-normalised Weil factor over $\mathcal{B}$}

The definition of the non-normalised Weil factor, achieved over fields in \cite[Sec. 1.1]{trias_theta1}, generalises to $\mathcal{A}$-algebras as explained below. Let $X$ be a vector space over $F$ of finite dimension $m$. Let $\mu^\mathcal{A}$ be an invertible Haar measure of $X$ with values in $\mathcal{A}$.

\begin{prop} Let $Q$ be a non-degenerate quadratic form on $X$. Then there exists a unique non-zero element $\Omega_{\mu^\mathcal{A}}(\psi^\mathcal{A} \circ Q)$ in $\mathcal{A}$ such that for all $f \in C_c^\infty(X,\mathcal{A})$, one has:
$$\int_X \int_X f(y-x) \psi^\mathcal{A}( Q(x)) d\mu^\mathcal{A} (x) d\mu^\mathcal{A} (y) = \Omega_{\mu^\mathcal{A}}(\psi^\mathcal{A} \circ Q) \int_X f(x) d \mu^\mathcal{A} (x).$$
For any sufficiently small open compact subgroup $K$ in $X$, the condition for smallness being ``$\psi^\mathcal{A} (Q(u))=1$ for all $u \in K$'', this factor explicitely reads:
$$\Omega_{\mu^\mathcal{A}}(\psi^\mathcal{A} \circ Q) = \sum_{\bar{x} \in K' / K} \psi^\mathcal{A}(Q(\bar{x}))$$
where $K' = \{ y \in X \ | \ \forall u \in K, \psi^\mathcal{A}(Q(y-u) -Q(y)) = 1 \}$ is a compact open subgroup too. \end{prop}

\begin{proof} The existence of such an element $\Omega_{\mu^\mathcal{A}}(\psi^\mathcal{A} \circ Q)$ comes from the definition of the non-normalised Weil factor over fields and from computation, as examined below. 

Indeed, the ring $\mathcal{A}$ is naturally contained in its field of fraction $\mathcal{K}$, and the measure $\mu^\mathcal{A}$ can be thought of as having values in $\mathcal{K}$. So there exists \cite[Prop. 1.2]{trias_theta1} a non-zero element $\Omega_{\mu^\mathcal{A}}(\psi^\mathcal{A} \circ Q)$ in $\mathcal{K}$, which achieves the first equality of the statement. A direct computation when $f=1_K$ and $\psi^\mathcal{A}(Q(K))=1$ gives:
$$\int_X 1_K(y-x) \psi^\mathcal{A}( Q(x)) d\mu^\mathcal{A} (x) = \psi^\mathcal{A}(Q(y)) \mu^\mathcal{A}(K) \times 1_{K'}(y)$$
where one easily checks from the definition that $K'$ is a compact open subgroup of $X$. In addition it contains $K$. Applying $\mu^\mathcal{A}$ to the previous equality leads to:
$$\Omega_{\mu^\mathcal{A}}(\psi^\mathcal{A} \circ Q)  \times \mu^\mathcal{A}(1_K)= \textup{vol}(K) \sum_{\bar{x} \in K' / K} \psi^\mathcal{A}(Q(\bar{x}))$$
where $\mu^\mathcal{A}(1_K) = \textup{vol}(K) \in \mathcal{A}^\times$ because $\mu$ is invertible, resulting in the last equality. \end{proof}

Let now $\mu$ be a Haar measure of $X$ with values in $\mathcal{B}$. Denote $\lambda_\mu$ the unique element in $\mathcal{B}$ such that $\mu = \lambda_\mu \times \mu^\mathcal{B}$, where $\mu^\mathcal{B} = \varphi \circ \mu^\mathcal{A}$ is an invertible Haar measure. Applying $\varphi$ to the equalities in the previous proposition yields:

\begin{cor} \label{weil_factor_over_B_def_cor} Let $Q$ be a non-degenerate quadratic form on $X$. Then there exists a unique element $\Omega_\mu(\psi^\mathcal{B} \circ Q)$ in $\mathcal{B}$ such that for all $f \in C_c^\infty(X,\mathcal{B})$, one has:
$$\int_X \int_X f(y-x) \psi^\mathcal{B}( Q(x)) d\mu(x) d\mu (y) = \Omega_{\mu}(\psi^\mathcal{B} \circ Q) \int_X f(x) d \mu (x).$$
Furthermore:
$$\Omega_{\mu}(\psi^\mathcal{B} \circ Q)= \lambda_\mu \times \varphi\left( \Omega_{\mu^\mathcal{A}}(\psi^\mathcal{A} \circ Q) \right).$$ \end{cor}

When $Q$ is a quadratic form on $X$, one denotes $\textup{rad}(Q)$ its radical. Observe that $Q$ is non-degenerate if and only if $\textup{rad}(Q) = 0$. The non-degenerate quadratic form $Q_{\textup{nd}}$ associated to $Q$ is the non-degenerate quadratic form induced by $Q$ on $X / \textup{rad}(Q)$.

\begin{defi} Let $Q$ be a quadratic form on $X$. Let $\mu$ be Haar measure of $X/\textup{rad}(Q)$ with values in $\mathcal{B}$. The non-normalised Weil factor is defined by:
\begin{itemize}
\item $\Omega_\mu(\psi^\mathcal{B} \circ Q) := \mu(\{0\})$ if $Q$ is the zero quadratic form;
\item $\Omega_\mu(\psi^\mathcal{B} \circ Q) : = \Omega_\mu (\psi^\mathcal{B} \circ Q_{\textup{nd}})$ otherwise. \end{itemize} \end{defi}

\begin{lem} One has :
$$\Omega_{\mu^\mathcal{A}}(\psi^\mathcal{A} \circ Q) \in \mathcal{A}^\times.$$
In particular for any invertible Haar measure $\mu$ with values in $\mathcal{B}$:
$$\Omega_\mu(\psi^\mathcal{B} \circ Q) \in \mathcal{B}^\times.$$ \end{lem}

\begin{proof} Let $\mathcal{K} \to \mathbb{C}$ be an embedding of $\mathcal{K}$ into $\mathbb{C}$ and $\varphi_\mathbb{C}$ its restriction to $\mathcal{A}$. The factor $\Omega_{\mu^\mathcal{A}}(\psi^\mathcal{A} \circ Q)$ can be thought of as the factor $\Omega_{\mu^\mathbb{C}}(\psi^\mathbb{C} \circ Q) = \varphi_\mathbb{C} (\Omega_{\mu^\mathcal{A}}(\psi^\mathcal{A} \circ Q))$ where $\mu^\mathbb{C}= \varphi_\mathbb{C} \circ \mu^\mathcal{A}$ is an invertible Haar measure. Then point f) of \cite[Prop. 1.5]{trias_theta1} gives:
$$\Omega_{\mu^\mathbb{C}}(\psi^\mathbb{C} \circ Q) = \omega_{\psi^\mathbb{C}}(\psi^\mathbb{C} \circ Q) \times |\rho|_{\mu^\mathbb{C}}^{\frac{1}{2}}$$
where $\omega_{\psi^\mathbb{C}}(\psi^\mathbb{C} \circ Q)$ is an $8$th root of unity and $|\rho|_{\mu^\mathbb{C}} = \mu^\mathbb{C}(K) (q^{\frac{1}{2}})^k$, with $K$ a compact open subgroup of $X$, a square root $q^{\frac{1}{2}}$ of $q$ in $\mathbb{C}$ and an integer $k \in \mathbb{Z}$. So:
$$\Omega_{\mu^\mathbb{C}}(\psi^\mathbb{C} \circ Q)^8= (\mu^\mathbb{C}(K))^8 q^{4k}.$$
Therefore $\Omega_{\mu^\mathcal{A}}(\psi^\mathcal{A} \circ Q)^8 =(\mu^\mathcal{A}(K))^8 q^{4k} \in \mathcal{A}^\times$ because $\varphi_\mathbb{C}$ is injective and $\mathbb{Q}$-linear, implying the result about the factor being invertible. Hence the second equality results from applying $\varphi$ and Corollary \ref{weil_factor_over_B_def_cor}, given the fact that $\lambda_\mu \in \mathcal{B}^\times$. \end{proof}

Define for $a$ in $F^\times$ the quadratic form $Q_a : x \in F \mapsto a x^2 \in F$. Then the factor:
$$\Omega_{a,b}^\mathcal{A} =  \frac{\Omega_{\mu^\mathcal{A}}(\psi^\mathcal{A} \circ Q_a)}{\Omega_{\mu^\mathcal{A}}( \psi^\mathcal{A} \circ Q_b)} \in \mathcal{A}^\times$$
does not depend on the choice of the invertible Haar measure $\mu^\mathcal{A}$, as the notation suggests. One can define $\Omega_{a,b}^\mathcal{B}$ in the obvious way, either as a quotient of two non-normalised Weil factors or as the image of the previous using the map $\varphi$.

\subsubsection{Section $\varsigma^\mathcal{B}$ giving the cocycle} \label{section_sigma_cocycle_section}

Let $X$ be a lagrangian of $W$. In particular this provides an instance of a self-dual subgroup in $W$. A nice section $\varsigma^\mathcal{A} : \textup{Sp}(W) \to \widetilde{\textup{Sp}}_{\psi^\mathcal{A},V_X^\mathcal{A}}^\mathcal{A}(W)$ of $p_\mathcal{A}$ is defined below. It is nice in the sense that it will give the explicit group law in the metaplectic group over $\mathcal{A}$.

First of all, observe that, using the notation of Section \ref{weil_over_A_algebra_section}, any section $\varsigma$ of $p_\mathcal{A}$ is given by a family $(\mu_g)_{g \in \textup{Sp}(W)}$ of measures where $\mu_g$ is an invertible measure of $g X \cap X \backslash X$. Namely it reads $\varsigma : g \mapsto (g,I_{gX,X,\mu_g,0} \circ I_g)$. One defines the section $\varsigma^\mathcal{A}$ mentioned above in the following way. The stabiliser $P(X)$ of $X$ in $\textup{Sp}(W)$ is a maximal parabolic subgroup. For $g \in \textup{Sp}(W)$, let $\mu_g$ be the invertible measure on $gX \cap X \backslash X$ defined by:
$$\mu_g = \Omega_{1,\textup{det}_X(p_1p_2)}^\mathcal{A} \times \phi_1 \cdot \mu_{w_j}^\mathcal{A}$$
where:
\begin{itemize}
\item $(w_j)_{j=0 \dots m}$ is a system of representatives in $\textup{Sp}(W)$ for $P(X) \backslash \textup{Sp}(W) / P(X)$;
\item the element $g = p_1 w_j p_2 \in P(X) w_j P(X)$ with $p_1$ and $p_2$ in $P(X)$;
\item $\textup{det}_X(p) = \textup{det}_F(p|_X)$ where $p|_X \in \textup{GL}(X) \simeq \textup{GL}_m(F)$;
\item $gX \cap X \backslash X \overset{\phi_1}{\simeq} w_j X \cap X \backslash X$ is induced by $x \in X \mapsto \overline{p_1^{-1} x} \in w_j X \cap X \backslash X$;
\item $Q_j(x) = \frac{1}{2} \langle w_j x , x \rangle$ is the non-degenerate quadratric form on $w_j X \cap X \backslash X$;
\item for any invertible $\mu$, set $\mu_{w_j}^\mathcal{A} = \Omega_\mu(\psi^\mathcal{A} \circ Q_j)^{-1} \mu$ which does not depend on $\mu$. \end{itemize}
See \cite[Sec. 3.5]{trias_theta1} to get a more detailed explanation about the previous definitions. Exclude the exceptional case $F=\mathbb{F}_3$ and $\textup{dim}(W) =2$ from now on.

\begin{prop} With the previous choice of $\mu_g$, the section: 
$$\varsigma^\mathcal{A} : g \in \textup{Sp}(W) \mapsto (g,I_{gX,X,\mu_g,0} \circ I_g) \in \widetilde{\textup{Sp}}_{\psi^\mathcal{A},V_X^\mathcal{A}}^\mathcal{A}(W)$$
has values in $\widehat{\textup{Sp}}_{\psi^\mathcal{A},V_X^\mathcal{A}}^\mathcal{A}(W)$, except in the exceptional case $F= \mathbb{F}_3$ and $\textup{dim}(W) = 2$. The $2$-cocycle defined by this section:
$$\hat{c}^\mathcal{A} : (g_1,g_2) \in \textup{Sp}(W) \times \textup{Sp}(W) \mapsto \varsigma^\mathcal{A}(g_1) \varsigma^\mathcal{A}(g_2) \varsigma^\mathcal{A}(g_1g_2)^{-1} \in \mathcal{A}^\times$$
is trivial when $F$ is finite, and has image $\{ \pm 1 \}$ when $F$ is local non-archimedean. \end{prop}

\begin{proof} Consider an embedding $\mathcal{K} \to \mathbb{C}$ and denote $\varphi_\mathbb{C}$ its restriction to $\mathcal{A}$. The map:
$$\widetilde{\phi_\mathbb{C}} : (g,M) \in  \widetilde{\textup{Sp}}_{\psi^\mathcal{A},V_A^\mathcal{A}}^\mathcal{A}(W) \to (g,\phi_\mathbb{C}(M)) \in \widetilde{\textup{Sp}}_{\psi^\mathbb{C},V_A^\mathbb{C}}^\mathbb{C}(W)$$
and the compatibility $\phi_\mathbb{C} (I_{gX,X,\mu^\mathcal{A},0}) = I_{gX,X,\mu^\mathbb{C},0}$ from Corollary \ref{intertwining_maps_models_A1_A2_cor} where $\mu^\mathbb{C}=\varphi_\mathbb{C} \circ \mu^\mathcal{A}$, leads to:
$$\widetilde{\phi_\mathbb{C}} \circ \varsigma^\mathcal{A} (g) = (g, I_{gX,X,\varphi_\mathbb{C} \circ \mu_g,0} \circ I_g).$$
But the measure $\varphi_\mathbb{C} \circ \mu_g$ above is the one defined in \cite[Lem. 3.23]{trias_theta1}, and according to \cite[Th. 3.27]{trias_theta1}, the map $\widetilde{\phi_\mathbb{C}} \circ \varsigma^\mathcal{A}$ is a section of $p_\mathbb{C}$ whose associated cocycle is trivial when $F$ is finite, and has values in the reduced metaplectic group when $F$ is local non-archimedean. The associated cocycle $\hat{c}^\mathbb{C}$ is trivial when $F$ is finite and has image $\{ \pm 1 \}$ when $F$ is local non-archimedean. Using point d) of Proposition \ref{metaplectic_group_over_A_prop}, the image of $\varsigma^\mathcal{A}$ lies in $\widehat{\textup{Sp}}_{\psi^\mathcal{A},V_X^\mathcal{A}}^\mathcal{A}(W)$, except in the exceptional case $F= \mathbb{F}_3$ and $\textup{dim}(W) =2$. In any case, the map $\varsigma^\mathcal{A}$ is injective so this defines a section of $p_\mathcal{A}$. In particular, it is a group morphism when $F$ is finite as a result of the cocycle $\hat{c}^\mathbb{C}$ being trivial. \end{proof}

One easily deduces from the previous proposition and Proposition \ref{metaplectic_group_over_B_prop}, the corollary:

\begin{cor} The section $\varsigma^\mathcal{B}= \widetilde{\phi_\mathcal{B}} \circ \varsigma^\mathcal{A}$ has values in $\widehat{\textup{Sp}}_{\psi^\mathcal{B},V_X^\mathcal{B}}^\mathcal{B}(W)$, except in the exceptional case $F= \mathbb{F}_3$ and $\textup{dim}(W) = 2$. The $2$-cocycle defined by this section:
$$\hat{c}^\mathcal{B} : (g_1,g_2) \in \textup{Sp}(W) \times \textup{Sp}(W) \mapsto \varsigma^\mathcal{B}(g_1) \varsigma^\mathcal{B}(g_2) \varsigma^\mathcal{B}(g_1g_2)^{-1} \in \mathcal{B}^\times$$
is trivial when $F$ is finite or $\textup{char}(\mathcal{B})=2$, and has image $\{ \pm 1 \}$ otherwise. \end{cor}

\begin{rem} \label{exceptional_case_sigmaA_sigmaB_rem} In the exceptional case, the section $\varsigma^\mathcal{A}$, resp. $\varsigma^\mathcal{B}$, can still be defined. However the derived group $[\textup{Sp}(W),\textup{Sp}(W)]$ is a strict subgroup of the symplectic group $\textup{Sp}(W)$. So the image of the previous sections, which are again group morphisms, is just a subgroup of the metaplectic group over $\mathcal{A}$, resp. over $\mathcal{B}$, that is isomorphic to $\textup{Sp}(W)$. \end{rem}

\section{Families of Weil representations}

Consider the map $\widehat{\phi_\mathcal{B}} : \widehat{\textup{Sp}}_{\psi^\mathcal{A},V_A^\mathcal{A}}^\mathcal{A}(W) \to \widetilde{\textup{Sp}}_{\psi^\mathcal{B},V_A^\mathcal{B}}^\mathcal{B}(W)$ of Definition \ref{reduced_met_group_hatA_to_tildeB_def}. The exceptional case $F=\mathbb{F}_3$ and $\textup{dim}(W)=2$ needs separate treatment, which will be done as a quick remark, so we exclude it from now on.

Let $H$ be a closed subgroup of $\textup{Sp}(W)$ and set:
$$\widetilde{H}^\mathcal{A} = p_\mathcal{A}^{-1}(H) \textup{ and } \widetilde{H}^\mathcal{B} = p_\mathcal{B}^{-1}(H).$$
Denote by $\widehat{H}^\mathcal{A}$ the intersection of $\widetilde{H}^\mathcal{A}$ and $\widehat{\textup{Sp}}_{\psi^\mathcal{A},V_A^\mathcal{A}}^\mathcal{A}(W)$. Recall that $\varphi : \mathcal{A} \to \mathcal{B}$ is the structure morphism of the $\mathcal{A}$-algebra $\mathcal{B}$ and consider the categories:
$$\textup{Rep}_\mathcal{B}'(\widehat{H}^\mathcal{A}) = \{ (\pi,V) \in \textup{Rep}_\mathcal{B}(\widehat{H}^\mathcal{A}) \ | \ \pi((\textup{Id}_W,\epsilon \textup{Id}_{V_A^\mathcal{A}}))=\varphi(\epsilon) \textup{Id}_V \textup{ for } \epsilon \in \{ \pm 1 \} \}$$
and:
$$\textup{Rep}_\mathcal{B}'(\widetilde{H}^\mathcal{B}) = \{ (\pi,V) \in \textup{Rep}_\mathcal{B}(\widetilde{H}^\mathcal{B}) \ | \ \pi((\textup{Id}_W,\lambda \textup{Id}_{V_A^\mathcal{B}}))=\lambda \textup{Id}_V \textup{ for } \lambda \in \mathcal{B}^\times \}.$$

\begin{prop} The functor:
$$(\pi,V) \in \textup{Rep}_\mathcal{B}'(\widetilde{H}^\mathcal{B})  \mapsto (\pi \circ \widehat{\phi_\mathcal{B}},V) \in \textup{Rep}_\mathcal{B}'(\widehat{H}^\mathcal{A})$$
defines an equivalence of category. \end{prop}

\begin{proof} This map is a functor and its inverse is given by the extension of scalars to $\mathcal{B}^\times$, that is for any $(\pi',V') \in \textup{Rep}_\mathcal{B}'(\widehat{H}^\mathcal{A})$, the representation:
$$\pi'' : (\hat{h},\lambda) \in \widehat{H}^\mathcal{A} \times \mathcal{B}^\times \mapsto \lambda \pi'(\hat{h}) \in \textup{GL}_\mathcal{B}(V')$$
factorises as a representation of $\widetilde{H}^\mathcal{B}$. Indeed, the surjective group morphism:
$$(\hat{h},\lambda) \in \widehat{H}^\mathcal{A} \times \mathcal{B}^\times \to \widehat{\phi_\mathcal{B}}(\hat{h}) \times (\textup{Id}_W,\lambda \textup{Id}_{V_A^\mathcal{B}}) \in \widetilde{H}^\mathcal{B}$$
is an isomorphism when $F$ is finite and has kernel $\{ ((\textup{Id}_W,\epsilon \textup{Id}_{V_A^\mathcal{A}}),\varphi(\epsilon)) \ | \ \epsilon \in \{ \pm 1 \} \}$ when $F$ is local non-archimedean. But $\textup{Ker}(\pi'')$ contains the kernel of the surjective map above, that is it factorises as claimed. \end{proof}

\begin{rem} \label{exceptional_case_weil_rep_rem} The reason for proving such a result is to consider the ``same'' group for any $\mathcal{A}$ -algebra $\mathcal{B}$, which is particularly convenient when looking at scalar extension for representations. For instance, the representation $\omega_{\psi^\mathcal{A},V_A^\mathcal{A}} \otimes_\mathcal{A} \mathcal{B} \in \textup{Rep}_\mathcal{B}'(\widehat{H}^\mathcal{A})$, which is the scalar extension of $\omega_{\psi^\mathcal{A},V_A^\mathcal{A}} \in \textup{Rep}_\mathcal{A}'(\widehat{H}^\mathcal{A})$, should be the ``same'' -- the proposition below making this ``same'' precise -- representation as 
$\omega_{\psi^\mathcal{B},V_A^\mathcal{B}} \in  \textup{Rep}_\mathcal{B}'(\widetilde{H}^\mathcal{B})$. \end{rem}

\begin{rem} In the exceptional case however, because the symplectic group $\textup{Sp}(W)$ is isomorphic to $\textup{SL}_2(\mathbb{F}_3)$, the derived group $\widehat{\textup{Sp}}_{\psi^\mathcal{A},V_A^\mathcal{A}}^\mathcal{A}(W)$ is a strict subgroup of the symplectic group. One needs to replace $\widehat{\phi_\mathcal{B}}$ by any morphism that embeds $\textup{Sp}(W)$ in the metaplectic group over $\mathcal{A}$, composed with $\widetilde{\phi_\mathcal{B}}$. One can take for example the embeddings $\varsigma^\mathcal{A}$ and $\varsigma^\mathcal{B}$ according to Remark \ref{exceptional_case_sigmaA_sigmaB_rem}. \end{rem}

From the previous proposition and Theorem \ref{weil_rep_compatib_PhiB_thm}, the following compatibility holds:

\begin{prop} \label{weil_rep_A_to_B_prop} The representations $\omega_{\psi^\mathcal{A},V_A^\mathcal{A}} \otimes_\mathcal{A} \mathcal{B}$ and $\omega_{\psi^\mathcal{B},V_A^\mathcal{B}}$ are isomorphic, in the sense that the canonical identification $V_A^\mathcal{A} \otimes_\mathcal{A} \mathcal{B}\simeq V_A^\mathcal{B}$ of Corollary \ref{metaplectic_representations_cor} induces an isomorphism in $\textup{Rep}_\mathcal{B}'(\widehat{H}^\mathcal{A})$, namely:
$$(\omega_{\psi^\mathcal{A},V_A^\mathcal{A}} \otimes_\mathcal{A} \mathcal{B},V_A^\mathcal{A} \otimes_\mathcal{A} \mathcal{B}) \simeq (\omega_{\psi^\mathcal{B},V_A^\mathcal{B}} \circ \widehat{\phi_\mathcal{B}},V_A^\mathcal{B}).$$ \end{prop}

Of course when $R$ is a field endowed with an $\mathcal{A}$-algebra structure, the representation $(\omega_{\psi^R,V_A^R},V_A^R)$ is the modular Weil representation on $W$ associated to $\psi^R$ and $V_A^R$, in the way they are defined in \cite[Chap. 2,II]{mvw} for $R=\mathbb{C}$ and in \cite[Sec. 3]{trias_theta1} for more general fields. Recall that in this situation $V_A^R$ is the metaplectic representation associated to $\psi^R$.

\paragraph{Dual pairs.} When $(H_1,H_2)$ is a dual pair in $\textup{Sp}(W)$, one may fix a model for the Weil representation and ``embeds'' the lift of the dual pairs in the derived subgroup of the metaplectic group over $\mathcal{A}$ through the natural multiplication map. One can also use the lifts in the metaplectic group over $\mathcal{A}$ instead of the derived subgroup. This means looking at the representation:
$$\omega_{\psi^\mathcal{B},V_A^\mathcal{B}} \circ \widehat{\phi_\mathcal{B}}|_{\widehat{H_1}^\mathcal{A} \times \widehat{H_2}^\mathcal{A}} \in \textup{Rep}_\mathcal{B}(\widehat{H_1}^\mathcal{A} \times \widehat{H_2}^\mathcal{A})$$
where the restriction $\widehat{H_1}^\mathcal{A} \times \widehat{H_2}^\mathcal{A} \to \widehat{\textup{Sp}}_{\psi^\mathcal{A},V_A^\mathcal{A}}^\mathcal{A}(W)$ is achieved using the natural multiplication map. Of course, when these lifts of dual pairs are split, one can always compose with their splittings to get representations of $H_1$ or $H_2$ themselves. It may happen that splittings do not exist in the derived subgroup even if they do exist in the metaplectic group itself \cite[Chap. 2, Rem. II.9]{mvw} and \cite[Sec. 4]{trias_theta1}. So one may switch hats for tildas depending on the dual pair one wants to consider.

\section{Features of the pair $(\textup{GL}_1(F),\textup{GL}_1(F))$} \label{features_GL1_GL1_section}

Suppose $F$ is a local non-archimedean field. Let $W$ be a symplectic space over $F$ of dimension $2$ and $W=X+Y$ be a complete polarisation. For $a \in F^\times$, define $m_a$ to be the unique endomorphism in $\textup{Sp}(W)$ such that in the previous basis:
$$m_a = \left[ \begin{array}{cl}
a & 0 \\
0 & a^{-1}
\end{array} \right].$$
The pair $(H_1,H_2)=(F^\times,F^\times)$ is defined by $(a_1,a_2) \mapsto m_{a_1} m_{a_2^{-1}}$. Up to some smooth characters of $H_1$ and $H_2$, the Weil representation $\omega_{H_1,H_2}$ is the ``geometric'' representation $(\rho,C_c^\infty(F,\mathcal{B}))$ where $H_1$ and $H_2$ act respectively on the left and on the right on the locally profinite space $F$. For $f \in C_c^\infty(F,\mathcal{B})$ and $a_1, a_2 \in F^\times$, it reads:
$$\rho(a_1,a_2) \cdot f : x \in F \mapsto  f(a_1^{-1} x a_2) \in \mathcal{B}.$$

\subsection{Level $0$ part}

The category $\textup{Rep}_\mathcal{B}(F^\times)$ is decomposed as a product of categories $\prod_{k \in \mathbb{N}} \textup{Rep}_\mathcal{B}^k(F^\times)$ where the index $k$ is also known as the level. Here the level $k$ subcategory can be defined using the biggest pro-$p$-subgroup $K$ of $F^\times$. Of course this group is $K=1 + \varpi_F \mathcal{O}_F$ for $\varpi_F$ a uniformiser in $F$. In addition the isomorphism $(k,u) \in \mathbb{Z} \times \mathcal{O}_F^\times \mapsto \varpi_F^k u \in F^\times$ induces an isomorphism $F^\times / K \simeq \mathbb{Z} \times (\mathbb{Z}/ (q-1) \mathbb{Z})$. Suppose from now on a choice of uniformiser $\varpi_F$ is made as well as a choice of a primitive $(q-1)$-root of unity $\zeta_{q-1}$ in $F$. Hence in the free part $\mathbb{Z}$ is generated by $\varpi_F$ and the torsion part $\mathbb{Z}/(q-1) \mathbb{Z}$ is generated by $\zeta_{q-1}$. So the group algebra $\mathcal{B}[F^\times/K]$ is isomorphic to the $\mathcal{B}$-algebra $\mathcal{B}[X^{\pm 1 },Z]/(Z^{q-1}-1)$, where $\varpi_F$ corresponds to $X$ and $\zeta_{q-1}$ to $Z$.

\paragraph{The level $0$ category.} As we are only interested in the level $0$ part, we shall only consider, for any $V \in \textup{Rep}_\mathcal{B}(F^\times)$, the direct factor representation $V^K$ made of $K$-invariant vectors. As for the representation $(\rho_l,C_c^\infty(F,\mathcal{B}))$ given by the left $F^\times$-action, this level~$0$ part is the subspace of bi-$K$-invariant functions:
$$C_c^\infty(F,\mathcal{B})^K = \{ f \in C_c^\infty(F,\mathcal{B}) \ | \ \forall x \in F \textup{ and } k \in K, f(xk)=f(kx)=f(x)\}.$$
In addition, the center $\mathfrak{z}^0$ of the level $0$ category $\textup{Rep}_\mathcal{B}^0(F^\times)$ is, because the group $F^\times$ is abelian, equal to the endomorphism ring of a minimal progenerator of $\textup{Rep}_\mathcal{B}^0(F^\times)$. Let $(\textbf{1}_K,\mathcal{B})$ be the free module $\mathcal{B}$ of rank $1$ with trivial $K$-action. Then $\textup{ind}_K^{F^\times} (\textbf{1}_K)$ is known to be a progenerator of $\textup{Rep}_\mathcal{B}^0(F^\times)$. As a space of function this also is $C_c^\infty(F^\times/K,\mathcal{B})$, which is a free module of rank $1$ over $\mathcal{B}[F^\times/K]$ generated by the characteristic function $1_K$. Therefore: 
$$\textup{End}_{F^\times}(\textup{ind}_K^{F^\times} (\textbf{1}_K) )=  \textup{End}_{\mathcal{B}[F^\times/K]}(\textup{ind}_K^{F^\times} (\textbf{1}_K)) \simeq \mathcal{B}[F^\times/K]$$
thanks to $\textup{ind}_K^{F^\times}(\textbf{1}_K)$ being free of rank $1$. So one can consider that the center $\mathfrak{z}^0$ is $\mathcal{B}[F^\times/K] \simeq \mathcal{B}[X^{\pm 1},Z]/(Z^{q-1}-1)$. Eventually, the level $0$ category is equivalent to the category of modules over the latter commutative ring.

\subsubsection{Specialisation using the center} \label{specialisation_using_the_center_section}

\paragraph{Morphism of the center.} Let $\mathcal{C}$ be a commutative $\mathcal{B}$-algebra. Let $\eta \in \textup{Hom}_{\mathcal{B}-\textup{alg}}(\mathfrak{z}^0,\mathcal{C})$ be a morphism of $\mathcal{B}$-algebras. Of course $\eta$ naturally endows $\mathcal{C}$ with a $\mathfrak{z}^0$-algebra structure. In addition, any representation in $\textup{Rep}_\mathcal{B}^0(F^\times)$ is canonically endowed with a $\mathfrak{z}^0$-module structure. By definition, this $\mathfrak{z}^0$-module structure commutes with the $F^\times$-action.

\begin{defi} For any $V \in \textup{Rep}_\mathcal{B}^0(F^\times)$, one defines the representation:
$$V_\eta = V \otimes_{\mathfrak{z}^0} \eta \in \textup{Rep}_\mathcal{C}(F^\times).$$ \end{defi}

\noindent \textbf{Examples.} Recall $\mathfrak{z}^0 = \mathcal{B}[X^{\pm 1},Z]/(Z^{q-1}-1)$. The following are easy claims:
\begin{itemize}
\item when $\mathcal{B}$ is a field and $\chi : F^\times/K \to \mathcal{B}^\times$ is a character, the $\mathcal{B}$-algebra morphism:
$$\eta_\chi : P \in \mathcal{B}[X^{\pm 1},Z]/(Z^{q-1}-1) \mapsto P(\chi(\varpi_F),\chi(\zeta)) \in \mathcal{B}$$
provides the biggest $\chi$-isotypic quotient $V_{\eta_\chi} = V_\chi$. Furthermore:
$$\textup{Ker}(\eta_\chi) = (X- \chi(\varpi_F),Z-\chi(\zeta)).$$
\item when $\varphi : \mathcal{B} \to \mathcal{B}'$ is a morphism of $\mathcal{B}$-algebras, the $\mathcal{B}$-algebra morphism:
$$\eta_\varphi : P \in \mathcal{B}[X^{\pm 1},Z]/(Z^{q-1}-1) \to \varphi(P) \in \mathcal{B}'[X^{\pm 1},Z]/(Z^{q-1}-1)$$
provides the extension of scalars $V_{\eta_\varphi} = V \otimes_{\mathcal{B}} \mathcal{B}'$. Furthermore:
$$\textup{Ker}(\eta_\varphi)=\textup{Ker}(\varphi) \cdot \mathfrak{z}^0.$$
\item let $\chi$ be a character with values in $\mathcal{B}^\times$, let $\mathfrak{m}$ a maximal ideal in $\mathcal{B}$, and denote $\varphi_\mathfrak{m}$ the quotient morphism $\mathcal{B} \to \mathcal{B}/\mathfrak{m}$ and $\chi_\mathfrak{m}=\varphi_\mathfrak{m} \circ \chi$, then: 
$$(V_{\eta_\chi})_{\eta_{\varphi_\mathfrak{m}}} = (V_{\eta_{\varphi_\mathfrak{m}}})_{\eta_{\chi_\mathfrak{m}}} \textup{ \textit{i.e.} } V_{\eta_\chi} \otimes_{\mathcal{B}} (\mathcal{B} / \mathfrak{m}) = (V \otimes_\mathcal{B} (\mathcal{B} / \mathfrak{m}) )_{\chi_\mathfrak{m}}.$$
Therefore the representation $V_{\eta_\chi}$ may be thought as a family of representations specialising at maximal ideals to biggest isotypic quotients, whereas it is less clear how direct methods would give a good definition of an isotypic quotient over a ring.
\end{itemize}

\begin{rem} Unlike the construction of the biggest isotypic quotient for irreducible representations with coefficients in a field, the natural map $V \mapsto V_\eta$ is not surjective in general. Of course if $\eta$ is surjective, the previous map is a quotient map. \end{rem}

\subsubsection{Isotypic families of the Weil representation} \label{families_weil_rep_subsubsection}

\paragraph{Level $0$ Weil representation.} Instead of considering representations with coefficients over different rings, this approach benefits from a greater flexbility when dealing with the level $0$ Weil representation $C_c^\infty(F,\mathcal{B})^K$. For example in the second situation with $\varphi \in \textup{Hom}_{\mathcal{B}-\textup{alg}}( \mathcal{B}, \mathcal{B}')$, and thanks to the description as spaces of functions, one has:
$$(C_c^\infty(F,\mathcal{B})^K)_{\eta_\varphi} = C_c^\infty(F,\mathcal{B}')^K.$$

\paragraph{Family for the trivial representation.} Set $V = C_c^\infty(F,\mathcal{B})^K$ and $V_0 = C_c^\infty(F^\times,\mathcal{B})^K$. Recall there is an exact sequence of representations, that is given by the function restriction to the closed set $\{0\}$ in $F$, namely:
$$0 \to V_0 \to V \to \textbf{1}_{F^\times}^\mathcal{B} \to 0$$
where $\textbf{1}_{F^\times}^\mathcal{B}$ is a free $\mathcal{B}$-module of rank $1$ endowed with the trivial $F^\times$-action. Consider now the ideal $I_{\textbf{1}} =(X-1,Z-1)$ in $\mathfrak{z}^0$ and the morphism $\eta_{\textbf{1}} : \mathfrak{z}^0 \to \mathfrak{z}^0/I_{\textbf{1}}$. As $(\textbf{1}_{F^\times}^\mathcal{B})_{\eta_{\textbf{1}}} = \textbf{1}_{F^\times}^\mathcal{B}$, and $V_0$ is free of rank $1$ over $\mathfrak{z}^0$, it induces an exact sequence:
$$\textbf{1}_{F^\times}^\mathcal{B} \to V_{\eta_{\textbf{1}}} \to \textbf{1}_{F^\times}^\mathcal{B} \to 0.$$
The kernel of the map $\mathcal{B} \to V_{\eta_{\textbf{1}}}$ is $(q-1) \mathcal{B}$ because:
$$((X-1) V +(Z-1) V )\cap V_0 = (X-1)V_0 + (Z-1) V_0 + (q-1) V_0.$$
So the following sequence is exact:
$$0 \to \textbf{1}_{F^\times}^{\mathcal{B}/(q-1) \mathcal{B}} \to V_{\eta_{\textbf{1}}} \to \textbf{1}_{F^\times}^\mathcal{B} \to 0.$$
Denoting by $\beta$ the image of $1_{\mathcal{O}_F}$ in $V_{\eta_1}$, the above sequence splits as $b \in \mathcal{B} \mapsto b \cdot \alpha \in V_{\eta_1}$ is a section of $V_{\eta_1} \to \textbf{1}_{F^\times}^\mathcal{B}$. So one has $ V_{\eta_{\textbf{1}}} \simeq \textbf{1}_{F^\times}^{\mathcal{B}/(q-1) \mathcal{B}} \oplus \textbf{1}_{F^\times}^\mathcal{B}$.

\paragraph{The family for $(X-q,Z-1)$.} It does not coincide with the family for the trivial representation, except at the non-banal prime ideals. These are the prime ideals $\mathcal{P}$ in $\mathcal{B}$ such that $\mathcal{P} \cap \mathbb{Z}$ is generated by a prime $\ell$ dividing $q-1$. Denoting $\eta$ the character $\mathfrak{z}^0 \to \mathfrak{z}^0/(X-q,Z-1)$ and $(\chi_\mathcal{B},\mathcal{B})$ the character such that $\chi_\mathcal{B}(\zeta) = 1$ and $\chi_\mathcal{B}(\varpi_F)=q$, one similarly has:
$$0 \to \chi_\mathcal{B} \to V_{\eta} \to \textbf{1}_{F^\times}^{\mathcal{B}/(q-1)\mathcal{B}} \to 0.$$
Indeed on the one hand $(X-1) \textbf{1}_{F^\times}^\mathcal{B} + (Z-1) \textbf{1}_{F^\times}^\mathcal{B} = (1-q) \textbf{1}_{F^\times}^\mathcal{B}$ so $(\textbf{1}_{F^\times}^\mathcal{B})_\eta = \textbf{1}_{F^\times}^{\mathcal{B} / (q-1) \mathcal{B}}$. On the other hand $V_0$ is $\mathfrak{z}^0$-free so $(V_0)_\eta = \chi_\mathcal{B}$. Denote by $\alpha$ and $\beta$ the images of $1_{1+\varpi_F \mathcal{O}_F}$ and $1_{\mathcal{O}_F}$ in $V_\eta$. The following computation helps identifying the $(q-1)$-torsion:
$$(X-1) \beta = (q-1) \alpha = (X-q) \beta + (1-q) \beta = (q-1)\beta.$$
Then $\lambda \in \mathcal{B} \mapsto \lambda (\beta - \alpha) \in V_\eta$ factorises as a section of $V_\eta \to \textbf{1}_{F^\times}^{\mathcal{B}/(q-1)\mathcal{B}}$. As a consequence, one has $V_\eta \simeq \textbf{1}_{F^\times}^{\mathcal{B}/(q-1)\mathcal{B}} \oplus \chi_\mathcal{B}$. 

\begin{rem} We interpret $\textbf{1}_{F^\times}^{\mathcal{B}/(q-1)\mathcal{B}}$ as the greatest common quotient of $\textbf{1}_{F^\times}^\mathcal{B}$ and $\chi_\mathcal{B}$. \end{rem}

\paragraph{More general families.} One can look at any ideal in $\mathfrak{z}^0$ to get more new families of representations. For example, instead of only looking at characters with values $1$ at $\zeta$, one can look at irreducible factors $Q$ of $Z^{q-1}-1$ that are different from $Z-1$, and consider the ideal $(P,Q)$ for an irreducible polynomial $P$ in $\mathcal{B}[X^{\pm 1}]$.

\begin{rem} Even when $\mathcal{B}$ is an integral domain, the previous classes of ideals $(P,Q)$ are not necessarily prime ideals in $\mathfrak{z}^0$. The irreducibility has therefore to be thought over the field extension $\textup{Frac}(\mathcal{B}) [Z] / (Q)$ of $\textup{Frac}(\mathcal{B})$ \textit{i.e.} $P$ is irreducible as a polynomial over this bigger field. Furthermore, letting $P$  be a non-unitary polynomial allows to consider characters with coefficients in $\textup{Frac}(\mathcal{B})$ \textit{e.g.} $\overline{\mathbb{Z}_\ell}[X^{\pm 1}] / (\ell X -1) = \overline{\mathbb{Q}_\ell}$ when $\mathcal{B}=\overline{\mathbb{Z}_\ell}$. \end{rem}

\subsection{Positive level part}

Let $k \in \mathbb{N}^*$. As a first observation, the level $k$ parts of the representations $C_c^\infty(F,\mathcal{B})$ and $C_c^\infty(F^\times,\mathcal{B})$ are equal. Therefore the problem reduces to understand the level $k$ part of the regular representation. The same techniques as in the previous paragraph apply once the center $\mathfrak{z}^k$ of the category has been made explicit. The study will not be developped in the present work for the sake of shortness. But in order to flag some differences, here are some remarks below:
\begin{itemize}
\item if $\mathcal{B}$ does not have enough $p$-power roots of unity, the situation is more complicated as no characters of level $k$ may exist, that is there does not exist a group morphism $ \chi : 1 + \varpi_F \mathcal{O}_F \to \mathcal{B}^\times$ such that $1 + \varpi_F^{k+1} \mathcal{O}_F \subset \textup{Ker}(\chi) \subsetneq 1 + \varpi_F^k \mathcal{O}_F$;
\item provided $\mathcal{B}$ as enough $p$-power roots of unity, the set of characters: 
$$\textup{Char}_\mathcal{B}^k = \{ \chi : 1 + \varpi_F \mathcal{O}_F \to \mathcal{B}^\times \ | \ \chi \in \textup{Rep}_\mathcal{B}^k(1+\varpi_F \mathcal{O}_F) \}$$
is not empty and decomposes the category $\textup{Rep}_\mathcal{B}^k(F^\times)$ as product of categories $\prod_{\chi \in \textup{Char}_\mathcal{B}^k} \textup{Rep}_\mathcal{B}^\chi(F^\times)$, where each category factor is equivalent to $\textup{Rep}_\mathcal{B}^0(F^\times)$.
\end{itemize}
In the first situation, the situation may be quite complicacted to write down, though this first situation only occurs when $F$ has positive characteristic. Indeed, $\mathcal{A}$ is isomorphic to $\mathbb{Z}[\frac{1}{p},\zeta_p]$ in this case, whereas it is $\mathbb{Z}[\frac{1}{p},\zeta_{p^\infty}]$ for characteristic zero $F$. In the event of $\mathcal{B}$ having enough $p$-power roots of unity, one can reduce the situation to the level $0$ part of $C_c^\infty(F^\times,\mathcal{B})$ as it is isomorphic to the $\chi$-part of $C_c^\infty(F^\times,\mathcal{B})$ for $\chi \in \textup{Char}_\mathcal{B}^k$. This latter has been studied in the previous section.

\bibliographystyle{alpha}
\bibliography{lesrefer}

\noindent \textsc{Departement of Mathematics, Imperial College London, United Kingdom} \\
\textit{Email address:} \texttt{jtrias@ic.ac.uk}

\end{document}